\documentclass[12pt,reqno]{amsart}

\usepackage{amssymb}

\newtheorem{theorem}{Theorem}[section]
\newtheorem{proposition}[theorem]{Proposition}
\newtheorem{lemma}[theorem]{Lemma}
\newtheorem{corollary}[theorem]{Corollary}

\theoremstyle{definition}

\newtheorem{remark}[theorem]{Remark}

\numberwithin{equation}{section}

\begin{document}

\title[A canonical connection on bundles on Riemann surfaces]{A canonical connection on bundles
on Riemann surfaces and Quillen connection on the theta bundle}

\author[I. Biswas]{Indranil Biswas}

\address{School of Mathematics, Tata Institute of Fundamental Research,
Homi Bhabha Road, Mumbai 400005}

\email{indranil@math.tifr.res.in}

\author[J. Hurtubise]{Jacques Hurtubise}

\address{Department of Mathematics, McGill University, Burnside
Hall, 805 Sherbrooke St. W., Montreal, Que. H3A 2K6, Canada}

\email{jacques.hurtubise@mcgill.ca}



\date{}

\begin{abstract}
We investigate the symplectic geometric and also the differential geometric aspects of the
moduli space of connections on a compact connected Riemann surface $X$. Fix a theta characteristic
$K^{1/2}_X$ on $X$; it defines a theta divisor on the moduli space ${\mathcal M}$ of
stable vector bundles on $X$ of rank $r$ degree zero.
Given a vector bundle $E\, \in\, {\mathcal M}$ lying outside
the theta divisor, we construct a natural holomorphic connection on $E$ that depends holomorphically
on $E$. Using this holomorphic connection, we construct a canonical holomorphic isomorphism between
the following two:
\begin{enumerate}
\item the moduli space $\mathcal C$ of pairs $(E,\, D)$, where $E\, \in\, {\mathcal M}$ and $D$ is a
holomorphic connection on $E$, and

\item the space ${\rm Conn}(\Theta)$ given by the sheaf of holomorphic connections on the line bundle
on $\mathcal M$ associated to the theta divisor.
\end{enumerate}
The above isomorphism between $\mathcal C$ and ${\rm Conn}(\Theta)$ is symplectic structure preserving, and
it moves holomorphically as $X$ runs over a holomorphic family of Riemann surfaces.
\end{abstract}

\maketitle

\tableofcontents

\section{Introduction}\label{sec0}

Let $X$ be a compact connected Riemann surface of genus at least two. Let $K^{1/2}_X$ be a square-root
of the canonical line bundle $K_X$ of $X$; it is called a theta characteristic of $X$. Let $\mathcal M$ denote
the moduli space of stable vector bundles on $X$ of rank $r$ and degree zero. It has the theta divisor
$D_\Theta$ defined
by all $E$ such that $H^0(X, \, E\otimes K^{1/2}_X)\, \not=\, 0$; the holomorphic
line bundle on $\mathcal M$ corresponding to the divisor
$D_\Theta$ is denoted by $\Theta$. The moduli space $\mathcal M$ has a natural K\"ahler
structure. The K\"ahler $2$-form on $\mathcal M$ coincides with the symplectic form on the ${\rm U}(r)$
character variety for $X$, \cite{Go}, \cite{AB}, once we identify this character variety with
$\mathcal M$ using \cite{NS} (see the map $\psi_U$ below).

Let $\mathcal C$ denote the moduli space of holomorphic connections on $X$ of rank $r$ such that the underlying
holomorphic vector bundle is stable; it projects to $\mathcal M$ by mapping elements to the
underlying holomorphic vector bundle. This $\mathcal C$ is a holomorphic torsor on $\mathcal M$ for the
holomorphic cotangent bundle $T^*\mathcal M$ (this means that the fibers of $T^*\mathcal M$ act freely
transitively on the fibers of $\mathcal C$ over $\mathcal M$). This moduli space $\mathcal C$ is equipped
with a natural holomorphic symplectic structure \cite{Go}, \cite{AB}. There is a natural $C^\infty$ section
$$
\psi_U\, :\, {\mathcal M}\, \longrightarrow\, \mathcal C
$$
that sends any $E\, \in\, {\mathcal M}$ to the unique unitary flat connection on $E$ \cite{NS}.

Let $\text{Conn}(\Theta)$ denote the holomorphic fiber bundle on ${\mathcal M}$ given by the
sheaf of holomorphic connections on the line bundle $\Theta$. There is a tautological holomorphic
connection on the pullback of $\Theta$ to $\text{Conn}(\Theta)$. The curvature of this tautological holomorphic
connection is a holomorphic symplectic form on $\text{Conn}(\Theta)$. This $\text{Conn}(\Theta)$ is also
a holomorphic torsor on $\mathcal M$ for $T^*\mathcal M$. Although the line bundle $\Theta$
depends on the choice of the theta characteristic $K^{1/2}_X$, the $T^*\mathcal M$--torsor
$\text{Conn}(\Theta)$ does not depend on the choice of the theta characteristic (see Remark \ref{rei}).

There is a unique Hermitian connection on
$\Theta$ whose curvature is the K\"ahler form on $\mathcal M$ \cite{Qu}. Let
$$
\psi_Q\, :\, {\mathcal M}\, \longrightarrow\, \text{Conn}(\Theta)
$$
be the corresponding $C^\infty$ section of the projection $\text{Conn}(\Theta)\, \longrightarrow\,\mathcal M$.

Since both $\mathcal C$ and $\text{Conn}(\Theta)$ are torsors over $\mathcal M$ for $T^*\mathcal M$, and they
are equipped with the $C^\infty$ sections $\psi_U$ and $\psi_Q$ respectively, there is a unique
$C^\infty$ isomorphism
$$
F\, :\, {\mathcal C} \, \longrightarrow\, \text{Conn}(\Theta)
$$
satisfying the following two conditions:
\begin{itemize}
\item $F$ takes the section $\psi_U$ to $\psi_Q$, and

\item $F$ preserves the $T^*\mathcal M$--torsor structure up to the multiplicative factor $2r$, meaning
$F(E,\, D+v)\,=\, F(E,\,D)+2r\cdot v$, where $E\, \in\, \mathcal M$ with $D$ a holomorphic connection on $E$
and $v\, \in\, T^*_E {\mathcal M}\,=\, H^0(X,\, \text{End}(E)\otimes K_X)$.
\end{itemize}

The following was proved in \cite{BH} (recalled here in Theorem \ref{thm1}):

\textit{The above isomorphism $F$ is holomorphic, and it preserves the holomorphic symplectic forms
up to the factor $2r$, meaning the pullback, by $F$, of the holomorphic symplectic form on $\text{Conn}(\Theta)$
coincides with $2r$ times the holomorphic symplectic form on $\mathcal C$.}

Take any holomorphic vector bundle $E\, \in\, {\mathcal M}$ such that $H^0(X, \, E\otimes K^{1/2}_X)\, =\, 0$ (so $E$ lies 
outside the theta divisor $D_\Theta$). We construct a natural holomorphic connection on $E$; see Section \ref{se3.1}.
Unlike the unitary connection, it moves holomorphically as $E$ moves in a holomorphic family of vector bundles. In 
fact, this connection moves holomorphically as the pair $(X,\, E)$ moves in a holomorphic family. Let
$$
\phi\, :\, {\mathcal M}\setminus D_\Theta \, \longrightarrow\,{\mathcal C}\big\vert_{{\mathcal M}\setminus D_\Theta}
$$
be the holomorphic section given by this natural holomorphic connection.

The holomorphic line $\Theta$ has a canonical trivialization outside the theta divisor $D_\Theta$. This
trivialization produces a holomorphic section of the fiber bundle $\text{Conn}(\Theta) \, \longrightarrow\,
\mathcal M$ outside $D_\Theta$. Let
$$
\tau\, :\, {\mathcal M}\setminus D_\Theta\, \longrightarrow\,\text{Conn}(\Theta)\big\vert_{{\mathcal M}\setminus
D_\Theta}
$$
be the section given by this canonical trivialization. Unlike the section $\psi_Q$, this section $\tau$ is holomorphic.

Since both ${\mathcal C}\big\vert_{{\mathcal M}\setminus D_\Theta}$ and
$\text{Conn}(\Theta)\big\vert_{{\mathcal M}\setminus
D_\Theta}$ are torsors over $\mathcal M\setminus D_\Theta$ for the holomorphic cotangent bundle
$T^*({\mathcal M}\setminus D_\Theta)$, and $\phi$ and $\tau$ are holomorphic sections, there is a unique
holomorphic isomorphism
$$
G\, :\, {\mathcal C}\big\vert_{{\mathcal M}\setminus D_\Theta} \, \longrightarrow\,
\text{Conn}(\Theta)\big\vert_{{\mathcal M}\setminus D_\Theta}
$$
satisfying the following two conditions:
\begin{itemize}
\item $G$ takes the section $\phi$ to $\tau$, and

\item $G$ preserves the $T^*({\mathcal M}\setminus D_\Theta)$--torsor structures up to the multiplicative
factor $2r$; this means that $G(E,\, D+v)\,=\, G(E,\,D)+2r\cdot v$, where $E\, \in\, {\mathcal M}\setminus D_\Theta$,
$D$ is a holomorphic connection on $E$ and $v\, \in\, T^*_E {\mathcal M}\,=\, H^0(X,\, \text{End}(E)\otimes K_X)$.
\end{itemize}

Our main result says the following (see Theorem \ref{thm2}):

\begin{theorem}\label{thm0}
The above isomorphism $G$ coincides with the restriction of the isomorphism $F$ to the open subset
${\mathcal C}\big\vert_{{\mathcal M}\setminus D_\Theta}$.
\end{theorem}

Theorem \ref{thm0} has the following consequence (see Corollary \ref{cor0}):

\begin{corollary}\label{cor-i}
The above holomorphic isomorphism $G$ extends to a holomorphic isomorphism
$$G'\, :\, {\mathcal C}\, \stackrel{\sim}{\longrightarrow}\, {\rm Conn}(\Theta)$$
over entire ${\mathcal M}$.
\end{corollary}

\begin{remark}\label{rem-i}
We note that the isomorphism $G$ in Theorem \ref{thm0} is constructed purely algebro-geometrically.
Hence the construction of its closure $G'$ in Corollary \ref{cor-i} is
purely algebro-geometric. On the other hand, the two $C^\infty$ sections $\psi_U$ and $\psi_Q$
mentioned earlier are not algebro-geometric. Theorem \ref{thm0} implies that given the
input of the algebro-geometric isomorphism $G'$, any one of the two sections
$\psi_U$ and $\psi_Q$ determines the other uniquely.
\end{remark}

As mentioned before, both ${\mathcal C}$ and $\text{Conn}(\Theta)$ are equipped with
holomorphic symplectic structures. Let $\Phi_1$ and $\Phi_2$ denote the holomorphic symplectic
forms on ${\mathcal C}$ and $\text{Conn}(\Theta)$ respectively.
We prove the following relationship between these two symplectic forms (see Corollary \ref{cor2}):

\begin{corollary}\label{cor-i2}
For the isomorphism $G'$ in Corollary \ref{cor-i},
$$(G')^*\Phi_2\,=\, 2r\cdot\Phi_1\, .$$
\end{corollary}

Both $\phi$ and $\tau$ move holomorphically as $X$ moves in a holomorphic family of Riemann surfaces.
Therefore, Theorem \ref{thm0} has the following consequence (see Proposition \ref{prop3}):

\begin{proposition}\label{pl1}
The isomorphism $F$ moves holomorphically as $X$ moves in a holomorphic family of Riemann surfaces.
\end{proposition}

It may be mentioned that Proposition \ref{pl1} not immediate from the isomorphism of \cite{BH}.

This is particularly useful as one of the questions inspiring this investigation is the omnipresence of the 
determinantal line in questions involving deformations of connections, that is isomonodromy; this manifests itself 
in the role of tau-functions. The role of this line is somewhat surprising; it is as if in a linear algebra problem, 
the main issue was the determinant. Several papers have been devoted to this issue, notably by Malgrange \cite{Ma}. 
This paper and its predecessor \cite{BH} can be viewed as a further exploration of this issue; one is comparing one 
torsor ($\mathcal C$), defined over the moduli space in terms of connections on a full Riemann surface, and another 
($\text{Conn}$) which is simply the natural locus for connections on the determinant line; the first should contain 
much more information, but for certain things, it does not.

\section{Moduli space of stable vector bundles}

\subsection{Two torsors on a moduli space}

Let $X$ be a compact connected Riemann surface of genus $g$, with $g\, \geq\, 2$. We 
recall that a holomorphic vector bundle $V$ on $X$ is called stable if
$$\frac{\text{degree}(F)}{\text{rank}(F)}\, <\, \frac{\text{degree}(V)}{\text{rank}(V)}$$ for all 
holomorphic subbundles $F\, \subsetneq\, V$ of positive rank. This condition implies that any stable vector
bundle is simple. Fix a positive integer $r$. Let ${\mathcal M}$ denote the moduli space 
of stable vector bundles on $X$ of rank $r$ and degree zero (see \cite{Ne}, \cite{Si1} for the construction
of this moduli space).

The holomorphic cotangent bundle of $X$ will be denoted by $K_X$. A holomorphic connection
on a holomorphic vector bundle $E$ on $X$ is a holomorphic differential operator of order one
$$
{\mathcal D}_E\, :\, E\, \longrightarrow\, E\otimes K_X
$$
satisfying the Leibniz identity, which says that
$D(f\cdot s)\,=\, f\cdot D(s) + s\otimes df$, where $s$ is any locally
defined holomorphic section of $E$ and $f$ is any locally
defined holomorphic function on $X$ \cite{At}. Holomorphic connections on $X$ are flat because
there are no nonzero $(2,\, 0)$--forms on a Riemann surface.

Let $\mathcal C$ denote the moduli space of all holomorphic
connections on $X$ of rank $r$ such that the underlying holomorphic vector bundle is stable
\cite{Si1}, \cite{Si2}. In
other words, $\mathcal C$ parametrizes all isomorphism classes of pairs of the form $(E,\,
{\mathcal D}_E)$, where $E$ is a stable holomorphic vector bundle on $X$ of rank $r$ and degree
zero and ${\mathcal D}_E$ is a holomorphic connection on $E$.

Any indecomposable holomorphic
vector bundle on $X$ of degree zero admits a holomorphic connection \cite[p.~203, Proposition
19]{At}, \cite{We}, in particular, any stable vector bundle on $X$ of degree zero admits a
holomorphic connection. Let
\begin{equation}\label{e1}
\varphi\, :\, {\mathcal C}\, \longrightarrow\, \mathcal M\, , \ \
(E,\, {\mathcal D}_E)\,\longmapsto\, E
\end{equation}
be the forgetful map that forgets the holomorphic connection; as noted above,
the map $\varphi$ is surjective. Any two holomorphic
connections on $E\, \in\, \mathcal M$ differ by an element of $$H^0(X,\, \text{End}(E)\otimes
K_X)\,=\, T^*_E\mathcal M\, .$$ In fact, $\mathcal C$ is a holomorphic torsor over $\mathcal M$
for the holomorphic cotangent bundle $T^*\mathcal M$. This means that there is a holomorphic action
$$
\delta\, :\, {\mathcal C}\times_{\mathcal M} T^*\mathcal M\, \longrightarrow\, {\mathcal C}
$$
of $T^*\mathcal M$ on $\mathcal C$ such that the map of fiber products
$$
{\mathcal C}\times_{\mathcal M} T^*\mathcal M\, \longrightarrow\, {\mathcal C}\times_{\mathcal M}{\mathcal C}\, ,
\ \ (a,\, b)\, \longmapsto\, (\delta(a,\, b),\, a) 
$$
is an isomorphism.

Any stable holomorphic vector bundle on $X$ of degree zero admits a unique holomorphic 
connection whose monodromy representation is unitary \cite[p.~560--561, Theorem 2]{NS}. 
Therefore, the projection $\varphi$ in \eqref{e1} has a $C^\infty$ section
\begin{equation}\label{e2}
\psi_U\, :\, {\mathcal M}\, \longrightarrow\, \mathcal C
\end{equation}
that sends any stable vector bundle
$E\, \in\, \mathcal M$ to the unique holomorphic connection on $E$ whose monodromy
representation is unitary. This section $\psi_U$ is \textit{not} holomorphic.

There is a natural holomorphic symplectic form on $\mathcal C$
\begin{equation}\label{sf}
\Phi_1\, \in\, H^0({\mathcal C},\, \Omega^2_{\mathcal C})
\end{equation}
\cite{Go}, \cite{AB}. It is known that this holomorphic $2$--form $\Phi_1$ is algebraic \cite{Bi2}.

We shall now construct another holomorphic torsor over $\mathcal M$
for the holomorphic cotangent bundle $T^*\mathcal M$.

Fix a theta characteristic $K^{1/2}_X$ on $X$. So $K^{1/2}_X$ is a holomorphic line bundle
on $X$ of degree $g-1$ such that $K^{1/2}_X\otimes K^{1/2}_X$ is holomorphically isomorphic to the
holomorphic cotangent bundle $K_X$. Fix a holomorphic isomorphism of 
$K^{1/2}_X\otimes K^{1/2}_X$ with $K_X$. Let
\begin{equation}\label{td}
D_\Theta\, :=\, \{E\, \in\, \mathcal M\, \mid\, H^0(X,\, E\otimes K^{1/2}_X)\,\not=\,0\}
\, \subset\, \mathcal M
\end{equation}
be the theta divisor on $\mathcal M$ (see \cite{La}). Note that by Riemann--Roch we have
$$
\dim H^0(X,\, E\otimes K^{1/2}_X) - \dim H^1(X,\, E\otimes K^{1/2}_X)\,=\,
\text{degree}(E\otimes K^{1/2}_X)-r(g-1)
$$
$$
=\, r(g-1)-r(g-1)\,=\, 0\, .
$$
So $H^1(X,\, E\otimes K^{1/2}_X)\,\not=\, 0$ if and only if $E\, \in\, D_\Theta$.
The holomorphic line bundle ${\mathcal O}_{\mathcal M}
(D_\Theta)$ on $\mathcal M$ will be denoted by $\Theta$.

Let $\text{At}(\Theta)\,\longrightarrow\, \mathcal M$ be the Atiyah bundle for $\Theta$. It
fits in the short exact sequence of holomorphic vector bundles
\begin{equation}\label{ef4}
0\,\longrightarrow\, {\mathcal O}_{\mathcal M}\,\longrightarrow\, \text{At}(\Theta)
\,\longrightarrow\,T{\mathcal M} \,\longrightarrow\, 0
\end{equation}
over $\mathcal M$ (see \cite[p.~187, Theorem 1]{At}). For $i\, \geq\, 0$, let
${\rm Diff}^1_{\mathcal M}(\Theta,\, \Theta)$ be the holomorphic vector bundle
on ${\mathcal M}$ given by the sheaf of holomorphic differential operators from
$\Theta$ to itself. We note that
$\text{At}(\Theta)\,=\, {\rm Diff}^1_{\mathcal M}(\Theta,\, \Theta)$, and the exact
sequence in \eqref{ef4} coincides with the sequence
$$
0\,\longrightarrow\, {\rm Diff}^0_{\mathcal M}(\Theta,\, \Theta)\,=\, {\mathcal O}_{\mathcal M}\,\longrightarrow
\, {\rm Diff}^1_{\mathcal M}(\Theta,\, \Theta) \,\longrightarrow\,T{\mathcal M} \,\longrightarrow\, 0\, ,
$$
where the projection to $T{\mathcal M}$ is the symbol map.
Consider the dual exact sequence of \eqref{ef4}
\begin{equation}\label{e4}
0\,\longrightarrow\, T^*{\mathcal M}\,\longrightarrow\, \text{At}(\Theta)^*
\,\stackrel{\alpha}{\longrightarrow}\,{\mathcal O}^*_{\mathcal M}\,=\,
{\mathcal O}_{\mathcal M} \,\longrightarrow\, 0\, .
\end{equation}
Let $1_{\mathcal M}\, :\, {\mathcal M}\,\longrightarrow\, {\mathcal O}_{\mathcal M}$ be
the section given by the constant function $1$ on $\mathcal M$. Now define
\begin{equation}\label{e3}
\text{At}(\Theta)^* \, \supset\, \alpha^{-1}(1_{\mathcal M}({\mathcal M}))\,=:\,
\text{Conn}(\Theta)\,\stackrel{q}{\longrightarrow}\,{\mathcal M} \, ,
\end{equation}
where $\alpha$ is the projection in \eqref{e4}. Holomorphic sections of $\text{Conn}(\Theta)$
over an open subset $U\, \subset\, {\mathcal M}$ are identified with the
holomorphic connections on $\Theta\big\vert_U$. From \eqref{e4} it follows immediately
that $\text{Conn}(\Theta)$ is a holomorphic torsor over $\mathcal M$
for the holomorphic cotangent bundle $T^*\mathcal M$.

For any open subset $U\, \subset\, {\mathcal M}$, all $C^\infty$ maps $s\, :\, U\, 
\,\longrightarrow\, \text{Conn}(\Theta)$ such that $q\circ s\,=\, \text{Id}_U$, where $q$ 
is the projection in \eqref{e3}, are in bijection with the $C^\infty$ connections $D_U$ on the 
holomorphic line bundle $\Theta\big\vert_U$ such that the $(0,\,1)$-part
of $D_U$ coincides with the Dolbeault 
operator given by the holomorphic structure on $\Theta\big\vert_U$. This condition on $D_U$ is
equivalent to the condition that
$D_U(\gamma)$ is a $C^\infty$ section of $(\Theta\otimes K_X)\big\vert_U$ for every holomorphic
section $\gamma$ of $\Theta_U$. Such a connection $D_U$ on 
$\Theta\big\vert_U$ is holomorphic if and only if the corresponding section $s$ of the 
projection $q$ is holomorphic.

There is a natural K\"ahler form $\omega_{\mathcal M}$ on $\mathcal M$ \cite{AB}; 
this form $\omega_{\mathcal M}$ coincides with the symplectic form on the irreducible unitary
character variety ${\rm Hom}^{\rm ir}(\pi_1(X),\, {\rm U}(r))/{\rm U}(r)$
that was constructed by Goldman \cite{Go}; here ${\rm Hom}^{\rm ir}(\pi_1(X),\, {\rm U}(r))$
denotes the space of all homomorphisms $\rho\,:\, \pi_1(X)\, \longrightarrow\, {\rm U}(r)$
such that the standard action of $\rho(\pi_1(X))$ on ${\mathbb C}^r$ does not preserve any
nonzero proper subspace of ${\mathbb C}^r$. More precisely, we have
\begin{equation}\label{kf}
\omega_{\mathcal M}\,=\, \psi^*_U\Phi_1\, ,
\end{equation}
where $\psi_U$ and $\Phi_1$ are as in \eqref{e2} and \eqref{sf} respectively.

Quillen constructed an explicit Hermitian structure on the line bundle $\Theta$ with the property that 
the curvature of the corresponding Chern connection on $\Theta$ coincides with the K\"ahler form 
$\omega_{\mathcal M}$ in \eqref{kf} \cite{Qu} (see also \cite{BGS1}, \cite{BGS2},
\cite{BGS3}). As a corollary, the de Rham cohomology class for 
$\omega_{\mathcal M}$ is integral. We note that there is at most one Hermitian connection
on $\Theta$ whose curvature is $\omega_{\mathcal M}$. In other words, the Chern connection
of the Hermitian structure on $\Theta$ constructed in \cite{Qu} is the unique Hermitian connection
whose curvature is $\omega_{\mathcal M}$. It should be clarified that this
condition --- that the curvature is $\omega_{\mathcal M}$ --- does not determine
the Hermitian structure on $\Theta$ uniquely; any two Hermitian structures on $\Theta$
satisfying this condition differ by a constant scalar multiplication. However,
the Hermitian connection is unique.
Let $\nabla^Q$ denote the unique Hermitian connection on 
$\Theta$ whose curvature is $\omega_{\mathcal M}$. So $\nabla^Q$ produces a $C^\infty$ section
\begin{equation}\label{e5}
\psi_Q\, :\, {\mathcal M}\, \longrightarrow\, \text{Conn}(\Theta)
\end{equation}
of the holomorphic fibration $q$ in \eqref{e3}.

There is a holomorphic symplectic form
\begin{equation}\label{sf2}
\Phi_2\, \in\, H^0(\text{Conn}(\Theta),\, \Omega^2_{\text{Conn}(\Theta)})
\end{equation}
on $\text{Conn}(\Theta)$ which can be described as follows. The holomorphic line bundle
$q^*\Theta$, where $q$ is the projection in \eqref{e3}, has a tautological holomorphic
connection (see \cite[p.~372, Proposition 3.3]{BHS}, \cite{BB}). The curvature of this
tautological holomorphic connection on $q^*\Theta$ is the $2$--form $\Phi_2$ in \eqref{sf2}.

\subsection{An isomorphism of torsors}

In this subsection a result from \cite{BH} will be recalled.

Let
\begin{equation}\label{mp}
\delta\, :\, {\mathcal C}\times_{\mathcal M} T^*{\mathcal M}\, \longrightarrow\,
{\mathcal C}\ \ \text{ and }\ \ \eta\, :\, {\rm Conn}(\Theta)\times_{\mathcal M}
T^*{\mathcal M}\, \longrightarrow\, {\rm Conn}(\Theta)
\end{equation}
be the holomorphic $T^*{\mathcal M}$--torsor structures on $\mathcal C$ and ${\rm Conn}(\Theta)$
respectively. Let
\begin{equation}\label{mq}
{\textbf m}\, :\, T^*{\mathcal M}\, \longrightarrow\,T^*{\mathcal M}\, ,\ \ v\, \longmapsto\, 2r\cdot v
\end{equation}
be the multiplication by $2r$.

\begin{theorem}[{\cite[Proposition 2.3, Theorem 3.1]{BH}}]\label{thm1}
There is a unique holomorphic isomorphism
$$
F\, :\, {\mathcal C}\, \longrightarrow\, {\rm Conn}(\Theta)
$$
such that
\begin{enumerate}
\item $\varphi\, =\, q\circ F$, where $\varphi$ and $q$ are the projections in
\eqref{e1} and \eqref{e3} respectively,

\item $F\circ\psi_U\,=\, \psi_Q$, where $\psi_U$ and $\psi_Q$ are the sections in
\eqref{e2} and \eqref{e5} respectively, and

\item $F\circ\delta \,=\, \eta \circ (F\times {\textbf m})$
as maps from ${\mathcal C}\times_{\mathcal M} T^*{\mathcal M}$
to ${\rm Conn}(\Theta)$, where $\delta$, $\eta$ and
${\textbf m}$ are the maps in \eqref{mp} and \eqref{mq}.
\end{enumerate}

Moreover, $F^*\Phi_2\,=\, 2r\cdot\Phi_1$, where $\Phi_1$ and $\Phi_2$ are the symplectic
forms in \eqref{sf} and \eqref{sf2} respectively.
\end{theorem}

There is a unique $C^\infty$ isomorphism ${\mathcal C}\, \longrightarrow\, {\rm Conn}(\Theta)$
that satisfies the three conditions in the first part of Theorem \ref{thm1}. The content of
the first part of Theorem \ref{thm1} is that this $C^\infty$ isomorphism is actually holomorphic.
The second part of Theorem \ref{thm1} says that this isomorphism is compatible, up to
the factor $2r$, with the symplectic structures on $\mathcal C$ and ${\rm Conn}(\Theta)$.

\begin{remark}\label{rei}
For a different choice of the theta characteristic on $X$, the corresponding theta line bundle on
${\mathcal M}$ differs from $\Theta$ by a line bundle of order two on ${\mathcal M}$. Any line
bundle of finite order has a tautological flat holomorphic connection. This implies that the
$T^*{\mathcal M}$--torsor ${\rm Conn}(\Theta)$ is actually independent of the choice that theta
characteristic on $X$.
\end{remark}

\section{Another isomorphism of torsors}

\subsection{A canonical connection}\label{se3.1}

For $i\,=\, 1,\, 2$, let
\begin{equation}\label{pi}
p_i\, :\, X\times X\, \longrightarrow\, X
\end{equation}
be the projection to the $i$--th factor. Let
$$
\Delta\, :=\, \{(x,\, x)\,\mid\, x\, \in\, X\}\, \subset\, X\times X
$$
be the diagonal divisor. We shall identify $\Delta$ with $X$ via the map
$x\, \longmapsto\, (x,\, x)$. Using the Poincar\'e adjunction formula,
the restriction of the holomorphic line bundle
${\mathcal O}_{X\times X}(\Delta)$ to $\Delta$ is identified with the normal
bundle of the divisor $\Delta\, \subset\, X\times X$, which in turn is identified with $TX$
using the identification of $\Delta$ with $X$. However this
isomorphism between ${\mathcal O}_{X\times X}(\Delta)\big\vert_\Delta$ and $TX$
changes by multiplication by $-1$ under the involution $(x,\, y)\, \longmapsto\,
(y,\, x)$ of $X\times X$. In other words, this involution acts by multiplication by $-1$ on
${\mathcal O}_{X\times X}(\Delta)\big\vert_\Delta$.

Using the isomorphism between ${\mathcal O}_{X\times X}(\Delta)\big\vert_\Delta$ and
$TX$, the restriction of $(p^*_1 K^{1/2}_X)\otimes (p^*_2 K^{1/2}_X) \otimes {\mathcal O}_{X\times X}(\Delta)$
to $\Delta$ is identified with $K_X\otimes TX\,=\, {\mathcal O}_X$. It should be
clarified that this isomorphism changes by multiplication by
$-1$ under the involution $(x,\, y)\, \longmapsto\, (y,\, x)$ of $X\times X$.

Take any $E\, \in\, {\mathcal M}\setminus D_\Theta$, where $D_\Theta$ is constructed in \eqref{td}. Since
$$
H^0(X,\, E\otimes K^{1/2}_X)\,=\,0\,=\, H^1(X,\, E\otimes K^{1/2}_X)\, ,
$$
using Serre duality, we have
\begin{equation}\label{e7}
H^0(X,\, E^*\otimes K^{1/2}_X)\,=\,H^1(X,\, E\otimes K^{1/2}_X)^*
\,=\,0\,=\, H^0(X,\, E\otimes K^{1/2}_X)^*\,=\, H^1(X,\, E^*\otimes K^{1/2}_X)\, .
\end{equation}

Consider the short exact sequence of coherent sheaves on $X\times X$
$$
0 \, \longrightarrow\, (p^*_1(E\otimes K^{1/2}_X))\otimes (p^*_2(E^*\otimes K^{1/2}_X)) \, \longrightarrow\,
(p^*_1(E\otimes K^{1/2}_X))\otimes (p^*_2(E^*\otimes K^{1/2}_X))
\otimes {\mathcal O}_{X\times X}(\Delta)
$$
$$
\longrightarrow\, \text{End}(E)\big\vert_\Delta\, \longrightarrow\, 0\, ;
$$
recall that $(p^*_1 K^{1/2}_X)\otimes (p^*_2 K^{1/2}_X) \otimes {\mathcal O}_{X\times X}(\Delta)\big\vert_\Delta\,=\,
{\mathcal O}_X$, and note that $(p^*_1 E)\otimes (p^*_2E^*)\big\vert_\Delta\,=\, \text{End}(E)$ (the
identification between $X$ and $\Delta$ is being used). Let
$$
H^0(X\times X,\, (p^*_1(E\otimes K^{1/2}_X))\otimes (p^*_2(E^*\otimes K^{1/2}_X)))
$$
$$
\longrightarrow\,
H^0(X\times X,\, (p^*_1(E\otimes K^{1/2}_X))\otimes (p^*_2(E^*\otimes K^{1/2}_X))
\otimes {\mathcal O}_{X\times X}(\Delta))
$$
\begin{equation}\label{e6}
\stackrel{h_E}{\longrightarrow}\, H^0(X,\, \text{End}(E))
\, \longrightarrow\, H^1(X\times X,\, (p^*_1(E\otimes K^{1/2}_X))\otimes (p^*_2(E^*\otimes K^{1/2}_X)))
\end{equation}
be the corresponding long exact sequence of cohomologies. Using K\"unneth formula,
$$
H^0(X\times X,\, (p^*_1(E\otimes K^{1/2}_X))\otimes (p^*_2(E^*\otimes K^{1/2}_X)))\,=\,
H^0(X,\, E\otimes K^{1/2}_X)\otimes H^0(X,\, E^*\otimes K^{1/2}_X)
$$
and
\begin{equation}\label{cv0}
H^1(X\times X,\, (p^*_1(E\otimes K^{1/2}_X))\otimes (p^*_2(E^*\otimes K^{1/2}_X)))
\end{equation}
$$
=\, \bigoplus_{j=0}^1
H^j(X,\, E\otimes K^{1/2}_X)\otimes H^{1-j}(X,\, E^*\otimes K^{1/2}_X)\, .
$$
Hence invoking \eqref{e7} we conclude that
\begin{equation}\label{cv}
H^k(X\times X,\, (p^*_1(E\otimes K^{1/2}_X))\otimes (p^*_2(E^*\otimes K^{1/2}_X)))\,=\, 0
\end{equation}
for $k\,=\, 0,\, 1$. Consequently, the homomorphism $h_E$ in the
exact sequence in \eqref{e6} is an isomorphism. Now define
\begin{equation}\label{e8}
\beta_E\, :=\, h^{-1}_E(\text{Id}_E)\, \in\, H^0(X\times X,\, (p^*_1(E\otimes K^{1/2}_X))\otimes (p^*_2(E^*\otimes
K^{1/2}_X))\otimes {\mathcal O}_{X\times X}(\Delta))\, ,
\end{equation}
where $\text{Id}_E\,\in\, H^0(X,\, \text{End}(E))$ is the identity endomorphism.

It was noted earlier that $(p^*_1 K^{1/2}_X)\otimes (p^*_2 K^{1/2}_X) \otimes
{\mathcal O}_{X\times X}(\Delta)\big\vert_\Delta\,=\,
{\mathcal O}_X$. We shall now show that
\begin{equation}\label{sh1}
((p^*_1 K^{1/2}_X)\otimes (p^*_2 K^{1/2}_X) \otimes {\mathcal O}_{X\times X}
(\Delta))\big\vert_{2\Delta}\,=\,{\mathcal O}_{2\Delta}\, .
\end{equation}

To prove \eqref{sh1}, take any holomorphic coordinate function $z\, :\, U\, \longrightarrow\, \mathbb C$ on
some analytic open subset $U$ of $X$. Take a holomorphic section
$$
s_z\, \in\, H^0(U,\, K^{1/2}_X\big\vert_U)
$$
such that $s_z\otimes s_z\,=\, dz\, \in\, H^0(U,\, K_U)$; note that there are exactly two
such sections, and they differ by multiplication by $-1$. Now we have
$$\frac{1}{z\circ p_1 - z\circ p_2}(p^*_1s_z)\otimes (p^*_2s_z)\,\in\, H^0(U\times U,\,
((p^*_1 K^{1/2}_X)\otimes (p^*_2 K^{1/2}_X) \otimes {\mathcal O}_{X\times X}(\Delta))\big\vert_{U\times U})\, .
$$
Of course, this section $\frac{1}{z\circ p_1 - z\circ p_2}(p^*_1s_z)\otimes (p^*_2s_z)$ depends 
on the coordinate function $z$. However, it is straight-forward to check that the restriction of 
the section to $(2\Delta)\bigcap (U\times U)$
$$\frac{1}{z\circ p_1 - z\circ p_2}(p^*_1s_z)\otimes (p^*_2s_z)\big\vert_{(2\Delta)\cap
(U\times U)}$$
is actually independent of the choice of the holomorphic coordinate function $z$. Consequently, the 
locally defined sections of the form $\frac{1}{z\circ p_1 - z\circ p_2}(p^*_1s_z)\otimes 
(p^*_2s_z)$ patch together compatibly to define a canonical section of $(p^*_1 K^{1/2}_X)\otimes 
(p^*_2 K^{1/2}_X) \otimes {\mathcal O}_{X\times X}(\Delta)\big\vert_{2\Delta}$. Let
\begin{equation}\label{se}
\sigma_0\, \in\, H^0(2\Delta,\, (p^*_1 K^{1/2}_X)\otimes (p^*_2 K^{1/2}_X)\otimes
{\mathcal O}_{X\times X}(\Delta))
\end{equation}
be this canonical section. This section $\sigma_0$ produces the isomorphism
in \eqref{sh1} between $((p^*_1 K^{1/2}_X)\otimes (p^*_2 K^{1/2}_X) \otimes
{\mathcal O}_{X\times X}(\Delta))\big\vert_{2\Delta}$ and ${\mathcal O}_{2\Delta}$ by sending
any locally defined section $f$ of ${\mathcal O}_{2\Delta}$ to the locally defined section $f\cdot\sigma_0$
of $((p^*_1 K^{1/2}_X)\otimes (p^*_2 K^{1/2}_X) \otimes
{\mathcal O}_{X\times X}(\Delta))\big\vert_{2\Delta}$.

We note that the restriction of the section $\sigma_0$ in \eqref{se} to $\Delta\, \subset\, 
2\Delta$ coincides with the section given by the trivialization of $((p^*_1 K^{1/2}_X)\otimes 
(p^*_2 K^{1/2}_X)\otimes{\mathcal O}_{X\times X}(\Delta))\big\vert_\Delta$ (the trivialization of 
$((p^*_1 K^{1/2}_X)\otimes (p^*_2 K^{1/2}_X)\otimes{\mathcal O}_{X\times 
X}(\Delta))\big\vert_\Delta$ was obtained earlier using the Poincar\'e adjunction formula). 
Consider the section $\beta_E$ in \eqref{e8}. It is easy to see that there is a unique section
$$
\widehat{\beta}_E\, \in\, H^0(2\Delta,\, (p^*_1 E)\otimes (p^*_2E^*))\, ,
$$
over $2\Delta$, such that
\begin{equation}\label{e9}
\beta_E\big\vert_{2\Delta}\,=\, \widehat{\beta}_E\otimes \sigma_0\, .
\end{equation}
Indeed, $(\sigma_0)^{-1}$ is a section of $((p^*_1 K^{1/2}_X)\otimes (p^*_2 K^{1/2}_X)\otimes
{\mathcal O}_{X\times X}(\Delta))^*$ over $2\Delta$. Now define
$$
\widehat{\beta}_E\,:=\, (\beta_E\big\vert_{2\Delta})\otimes (\sigma_0)^{-1}\, ,
$$
and consider it as a section of $((p^*_1 E)\otimes (p^*_2E^*))\big\vert_{2\Delta}$ using the
duality pairing
$$
(((p^*_1 K^{1/2}_X)\otimes
(p^*_2 K^{1/2}_X)\otimes{\mathcal O}_{X\times X}(\Delta))\big\vert_{2\Delta})\otimes
(((p^*_1 K^{1/2}_X)\otimes (p^*_2 K^{1/2}_X)\otimes
{\mathcal O}_{X\times X}(\Delta))^*\big\vert_{2\Delta})
\, \longrightarrow\,{\mathcal O}_{2\Delta}\, .
$$

Since $h_E(\beta_E) \,=\, \text{Id}_E$ (see \eqref{e8}), and the restriction of
$\sigma_0$ to $\Delta\, \subset\, 2\Delta$ coincides with the section
of $((p^*_1 K^{1/2}_X)\otimes (p^*_2 K^{1/2}_X)\otimes{\mathcal O}_{X\times
X}(\Delta))\big\vert_\Delta$ given by its trivialization, we conclude that
$$
\widehat{\beta}_E\big\vert_\Delta\,=\, \text{Id}_E
$$
using the natural identification of $((p^*_1 E)\otimes (p^*_2E^*))\big\vert_\Delta\,\longrightarrow
\,\Delta$ with $\text{End}(E)\,\longrightarrow\, X$. Consequently, the section $\widehat{\beta}_E$
in \eqref{e9} defines a holomorphic connection on the holomorphic vector
bundle $E$, following the idea of Grothendieck of defining a connection as an extension, to the first order
neighborhood of the diagonal, of the isomorphism of the
two pullbacks on the diagonal (see \cite[p.~6, 2.2.4]{De}). This holomorphic connection on $E$ defined by
$\widehat{\beta}_E$ will be denoted by
\begin{equation}\label{be}
\widehat{\beta}'_E\, .
\end{equation}

For notational convenience, let
\begin{equation}\label{nc}
{\mathcal M}^0\, :=\, {\mathcal M}\setminus D_\Theta
\end{equation}
denote the complement of $D_\Theta$ in ${\mathcal M}$. Define
$$
{\mathcal C}^0\,:=\, \varphi^{-1}({\mathcal M}^0)\, \subset\, {\mathcal C}\, ,
$$
where $\varphi$ is the projection in \eqref{e1}. Let
\begin{equation}\label{c0}
\widehat{\varphi}\, :\, {\mathcal C}^0\, \longrightarrow\,{\mathcal M}^0
\end{equation}
be the restriction of the map $\varphi$ to ${\mathcal C}^0$. We note that ${\mathcal C}^0$ is a
holomorphic torsor over ${\mathcal M}^0$ for $T^*{\mathcal M}^0$.

We have the holomorphic map
\begin{equation}\label{c1}
\phi\, :\, {\mathcal M}^0\, \longrightarrow\,{\mathcal C}^0\, ,\ \ E\, \longmapsto\, (E,\, \widehat{\beta}'_E)\, ,
\end{equation}
where $\widehat{\beta}'_E$ is the holomorphic connection
in \eqref{be}. So $\phi$ is a holomorphic
section of the projection $\widehat{\varphi}$ in \eqref{c0}, meaning $\widehat{\varphi}\circ\phi\,=\, \text{Id}_{{\mathcal M}^0}$.

\subsection{A holomorphic isomorphism of torsors}

Define
$$
\text{Conn}(\Theta)^0\, :=\, q^{-1}({\mathcal M}^0)\, \subset\, \text{Conn}(\Theta)\, ,
$$
where $q$ is the projection in \eqref{e3}, and ${\mathcal M}^0$ is the Zariski open subset of $\mathcal M$ in
\eqref{nc}. Let
\begin{equation}\label{c0c}
\widehat{q}\, :\, \text{Conn}(\Theta)^0\, \longrightarrow\,{\mathcal M}^0
\end{equation}
be the restriction of the map $q$ to $\text{Conn}(\Theta)^0$. We note that $\text{Conn}(\Theta)^0$
is a holomorphic torsor over ${\mathcal M}^0$ for $T^*{\mathcal M}^0$.

The restriction of the
line bundle $\Theta\,=\, {\mathcal O}_{\mathcal M}(D_\Theta)$ to ${\mathcal M}^0$ has a tautological 
isomorphism with the trivial line bundle ${\mathcal O}_{{\mathcal M}^0}$. Therefore, the 
trivial holomorphic connection on ${\mathcal O}_{{\mathcal M}^0}$, defined
by the de Rham differential, produces a holomorphic 
connection on the restriction $\Theta\big\vert_{{\mathcal M}^0}$. Let
\begin{equation}\label{c1c}
\tau\, :\, {\mathcal M}^0\, \longrightarrow\,\text{Conn}(\Theta)^0
\end{equation}
be the holomorphic section of the projection $\widehat{q}$ in \eqref{c0c} given by this tautological connection
on $\Theta\big\vert_{{\mathcal M}^0}$.

Let
\begin{equation}\label{mp2}
\delta^0\, :\, {\mathcal C}^0\times_{{\mathcal M}^0} T^*{\mathcal M}^0\, \longrightarrow\,
{\mathcal C}^0\ \ \text{ and }\ \ \eta^0\, :\, {\rm Conn}(\Theta)^0\times_{{\mathcal M}^0}
T^*{\mathcal M}^0\, \longrightarrow\, {\rm Conn}(\Theta)^0
\end{equation}
be the restrictions of the maps $\delta$ and $\eta$ in \eqref{mp}. So $\delta^0$ and
$\eta^0$ give the $T^*{\mathcal M}^0$--torsor structures on ${\mathcal C}^0$ and ${\rm Conn}(\Theta)^0$
respectively. Similarly,
\begin{equation}\label{mq2}
{\textbf m}^0\, :\, T^*{\mathcal M}^0\, \longrightarrow\,T^*{\mathcal M}^0\, ,\ \ v\, \longmapsto\, 2r\cdot v
\end{equation}
is the restriction of the map in \eqref{mq}.

\begin{lemma}\label{lem1}
There is a unique holomorphic isomorphism
$$
G\,\, :\,\, {\mathcal C}^0\, \longrightarrow\, {\rm Conn}(\Theta)^0
$$
such that
\begin{enumerate}
\item $\widehat{\varphi}\, =\, \widehat{q}\circ G$, where $\widehat{\varphi}$ and
$\widehat{q}$ are the projections in \eqref{c0} and \eqref{c0c} respectively,

\item $G\circ\phi\,=\, \tau$, where $\phi$ and $\tau$ are the sections in
\eqref{c1} and \eqref{c1c} respectively, and

\item $G\circ\delta^0 \,=\, \eta^0 \circ (G\times {\textbf m}^0)$ as maps from
${\mathcal C}^0\times_{{\mathcal M}^0} T^*{\mathcal M}^0$ to ${\rm Conn}(\Theta)^0$, where
$\delta^0$, $\eta^0$ and ${\textbf m}^0$ are the maps in \eqref{mp2} and \eqref{mq2}.
\end{enumerate}
\end{lemma}

\begin{proof}
This is straightforward. For any stable vector bundle $E\, \in\, {\mathcal M}^0$ and any $$\nu\,\in\, H^0(X,\, \text{End}(E)
\otimes K_X)\,=\, T^*_E{\mathcal M}^0,$$ define
$$
G(\delta^0(\phi(E),\, \nu))\,=\, \eta^0(\tau(E),\, 2r\cdot\nu)\, .
$$
Then $G$ is evidently a well defined map from ${\mathcal C}^0$ to ${\rm Conn}(\Theta)^0$. It is
holomorphic because $\phi$, $\tau$, $\delta^0$ and $\eta^0$ are all holomorphic maps. This map $G$
satisfies all the three conditions in the lemma. The uniqueness of $G$ is evident.
\end{proof}

\section{The two isomorphisms of torsors coincide}

The following theorem is the main result proved here.

\begin{theorem}\label{thm2}
The restriction of the isomorphism $F$ in Theorem \ref{thm1} to the open subset
${\mathcal C}^0\,\subset\, {\mathcal C}$ coincides with the isomorphism $G$ in
Lemma \ref{lem1}.
\end{theorem}

\begin{proof}
In view of the first condition in both Theorem \ref{thm1} and Lemma \ref{lem1},
we get a map
\begin{equation}\label{ga0}
\Gamma_0\, :\, {\mathcal C}^0\, \longrightarrow\, T^*{\mathcal M}^0\, ,\ \ z\, \longmapsto\,
F(z)- G(z)\, .
\end{equation}
In other words, $F(z)\,=\,\eta^0(G(z),\, \Gamma_0(z))$.
This map $\Gamma_0$ is holomorphic because both $F$ and $G$ are so. Take any $E\, \in\,
{\mathcal M}^0$ and any $$\alpha\, \in\, \widehat{\varphi}^{-1}(E)\, \subset\,{\mathcal C}^0\, ,
$$
where $\widehat{\varphi}$
is the projection in \eqref{c0}, and also take any
$$\nu\, \in\, H^0(X,\, \text{End}(E) \otimes K_X)\,=\, T^*_E{\mathcal M}^0\, .$$ Now from the
third condition in both Theorem \ref{thm1} and Lemma \ref{lem1}, we have
$$
\Gamma_0(\alpha+\nu)\,=\, F(\alpha+\nu)- G(\alpha+\nu)\,=\, F(\alpha)- G(\alpha)+2r\cdot\nu-2r\cdot\nu
\,=\, \Gamma_0(\alpha)\, .
$$
Consequently, the map $\Gamma_0$ in \eqref{ga0} produces a holomorphic $1$--form
\begin{equation}\label{ga}
\Gamma\, \in\, H^0({\mathcal M}^0,\, T^*{\mathcal M}^0)
\end{equation}
that sends any $E\, \in\, {\mathcal M}^0$ to $\Gamma_0(\alpha)\,\in\, T^*_E{\mathcal M}^0$
with $\alpha\, \in\, \widehat{\varphi}^{-1}(E)$; as shown above, $\Gamma_0(\alpha)$ is independent
of the choice of $\alpha$. In other words, $\Gamma_0\, =\,\Gamma\circ\widehat{\varphi}$,
where $\widehat{\varphi}$ is the projection in \eqref{c0}.

The following proposition would be used in the proof of Theorem \ref{thm2}.

\begin{proposition}\label{prop1}
The holomorphic $1$--form $\Gamma$ on ${\mathcal M}^0$ in \eqref{ga} is a meromorphic
$1$--form on ${\mathcal M}$, and its order of pole at the divisor
$D_\Theta\,=\, {\mathcal M}\setminus {\mathcal M}^0$ is at most one, or equivalently,
$$
\Gamma\, \in\, H^0({\mathcal M},\, (T^*{\mathcal M})\otimes\Theta)\,=\, H^0({\mathcal M},\,
(T^*{\mathcal M})\otimes {\mathcal O}_{\mathcal M}(D_\Theta))\, .
$$
\end{proposition}

\begin{proof}[{Proof of Proposition \ref{prop1}}]
Let ${\mathbb W}\, \longrightarrow\, {\mathcal M}$ be a holomorphic torsor for
$T^*{\mathcal M}$ and $s$ a holomorphic section of ${\mathbb W}$ over the open
subset ${\mathcal M}^0\,=\, {\mathcal M}\setminus D_\Theta$. Then the meromorphicity of
$s$ is defined by choosing holomorphic trivializations of ${\mathbb W}$
on open neighborhoods, in ${\mathcal M}$, of points of $D_\Theta$ (a trivialization
of a torsor is just a holomorphic section of it). Such a trivialization of $\mathbb W$
over $U\, \subset\, {\mathcal M}$ turns $s$ into a holomorphic $1$--form on $U\cap {\mathcal M}^0$;
define $s\big\vert_U$
to be meromorphic if this holomorphic $1$--form on $U\cap {\mathcal M}^0$ is meromorphic
near $D_\Theta\bigcap U\, \subset\, U$. Since
any two holomorphic trivializations, over $U$, of the torsor ${\mathbb W}$ differ by a holomorphic
$1$--form on $U$, this definition of the meromorphicity of $s\big\vert_U$ does not depend on the
choice of the trivialization of ${\mathbb W}\big\vert_U$. For the same reason, the order
of pole at $D_\Theta$ of a meromorphic section $s$ of ${\mathbb W}$ of the above type is also well-defined.

Let $\varpi_1$ be the smooth $(1,0)$--form on ${\mathcal M}^0$ given by $\phi-\psi_U\big\vert_{{\mathcal M}^0}$, where
$\phi$ (respectively, $\psi_U$) is the section of the $T^*{\mathcal M}^0$--torsor ${\mathcal C}^0$
(respectively, $T^*{\mathcal M}$--torsor ${\mathcal C}$)
constructed in \eqref{c1} (respectively, \eqref{e2}). Let $\varpi_2$ be the
smooth $(1,0)$--form on ${\mathcal M}^0$
given by $\tau-\psi_Q\big\vert_{{\mathcal M}^0}$, where
$\tau$ (respectively, $\psi_Q$) is the section of the $T^*{\mathcal M}^0$--torsor $\text{Conn}(\Theta)^0$
(respectively, $T^*{\mathcal M}$--torsor $\text{Conn}(\Theta)$) constructed
in \eqref{c1c} (respectively, \eqref{e5}). It can be shown that
\begin{equation}\label{fi}
\Gamma\,=\, 2r\cdot\varpi_1- \varpi_2\, .
\end{equation}
Indeed, using the third property in Theorem \ref{thm1} and the third property in Lemma \ref{lem1} we have
$$2r\cdot\varpi_1- \varpi_2\, =\, 2r(\phi-\psi_U\big\vert_{{\mathcal M}^0})
- (\tau-\psi_Q\big\vert_{{\mathcal M}^0})
\,=\, G(\phi)-G(\psi_U\big\vert_{{\mathcal M}^0})- \tau +\psi_Q\big\vert_{{\mathcal M}^0}
$$
$$
=\, \psi_Q\big\vert_{{\mathcal M}^0} -G(\psi_U\big\vert_{{\mathcal M}^0})
\,=\, F(\psi_U\big\vert_{{\mathcal M}^0}) - G(\psi_U\big\vert_{{\mathcal M}^0})
\,=\, \Gamma\, .
$$

Both $\psi_U$ and $\psi_Q$ are smooth sections over entire ${\mathcal M}$. The holomorphic section $\tau$ of 
$\text{Conn}(\Theta)^0$ is a meromorphic section of $\text{Conn}(\Theta)$ with a pole of order one at $D_\Theta$. 
Indeed, this follows immediately from the fact that the holomorphic connection on the line bundle 
$\Theta\big\vert_{{\mathcal M}^0} \,=\, {\mathcal O}_{{\mathcal M}^0}$ over ${\mathcal M}^0$, given by the 
canonical holomorphic trivialization of $\Theta\big\vert_{{\mathcal M}^0}$ (the holomorphic connection is defined 
by the de Rham differential), is actually a logarithmic connection on $\Theta$ over ${\mathcal M}$. In view of 
these, using \eqref{fi} we conclude the following:
\begin{itemize}
\item $\Gamma$ is a meromorphic $1$--form on ${\mathcal M}$ if and only if the section $\phi$ of ${\mathcal C}^0$ in 
\eqref{c1} is meromorphic, and

\item if $\Gamma$ is meromorphic, and the order of its pole at $D_\Theta$ is more than one,
then the order of the pole of $\Gamma$ at $D_\Theta$ coincides with the order of pole of $\phi$ at 
$D_\Theta$, in particular, the order of the pole of $\phi$ at $D_\Theta$ is more than one.
\end{itemize}

Therefore, to prove the proposition it suffices to show the following two:
\begin{enumerate}
\item the section $\phi$ of 
${\mathcal C}^0$ is meromorphic, and

\item the order of pole of $\phi$ at $D_\Theta$ is one.
\end{enumerate}
These will be proved by giving a global construction of $\phi$.

It is known that there is no Poincar\'e vector bundle over $X\times {\mathcal M}$
\cite[p.~69, Theorem 2]{Ra}. However, there is a canonical algebraic vector bundle over $X\times
X\times {\mathcal M}$ whose fiber over $X\times X\times\{E\}$ is $E\boxtimes E^*\,=\, (p^*_1E)
\otimes (p^*_2E^*)$ for every $E\, \in\, {\mathcal M}$, where $p_1$ and $p_2$ are the
projections in \eqref{pi}. This canonical
vector bundle on $X\times X\times {\mathcal M}$, which we shall denote by ${\mathcal E}$, can be
constructed as a descended bundle from the product of $X\times X$ with the quot scheme. The
reason that the corresponding vector bundle descends is that the action of the multiplicative
group ${\mathbb C}^*$ on $E\boxtimes E^*$, induced
by the scalar multiplications on $E$, is the trivial action. The restriction of
the vector bundle ${\mathcal E}\, \longrightarrow\, X\times X\times {\mathcal M}$ to
$$
\Delta\times {\mathcal M}\, \subset\, X\times X\times {\mathcal M}
$$
coincides with the universal endomorphism bundle over $X\times {\mathcal M}$. Let
\begin{equation}\label{cV}
{\mathcal V}\,:=\, {\mathcal E}\big\vert_{\Delta\times {\mathcal M}}\, \longrightarrow\, 
\Delta\times {\mathcal M}\,=\, X\times {\mathcal M}
\end{equation}
be the universal endomorphism bundle. So we have ${\mathcal V}\big\vert_{X\times\{E\}}\,=\, \text{End}(E)$
for all $E\, \in\, \mathcal M$.

Let $q_{12}\, :\, X\times X\times {\mathcal M}\, \longrightarrow\, X\times X$ be the projection to the
first two factors in the Cartesian product. Let
$$
q_{2}\, :\, X\times X\times {\mathcal M}\, \longrightarrow\, X\, , \ \ (x,\, y,\, E)\, \longmapsto\, y
$$
be the projection to the second factor. Let
\begin{equation}\label{J}
J\, :\, X\times X\times {\mathcal M}\, \longrightarrow\, {\mathcal M}\, , \ \ (x,\, y,\, E)\, \longmapsto\, E
\end{equation}
be the projection to the third factor. For notational convenience, the holomorphic line bundle
$q_{12}^*((p^*_1 K^{1/2}_X) \otimes (p^*_2K^{1/2}_X))$ on $X\times X\times {\mathcal M}$ will be denoted
by $\mathcal K$; recall that $K^{1/2}_X$ is a theta characteristic on $X$.

Consider the vector bundle
\begin{equation}\label{cvb}
{\mathcal E}\otimes {\mathcal K}\otimes q^*_{12}{\mathcal O}_{X\times X}(\Delta)\,
\longrightarrow\, X\times X\times {\mathcal M}\, .
\end{equation}
It fits in the following short exact sequence of coherent sheaves on $X\times X\times {\mathcal M}$:
\begin{equation}\label{ls}
0\, \longrightarrow\,{\mathcal E}\otimes {\mathcal K}\, 
\longrightarrow\,{\mathcal E}\otimes {\mathcal K}\otimes q^*_{12} {\mathcal 
O}_{X\times X}(\Delta) \, \longrightarrow\, {\mathcal V} \, 
\longrightarrow\, 0\, ,
\end{equation}
where $\mathcal V$ is defined in \eqref{cV}, and it is supported on
$\Delta\times{\mathcal M}\,=\, X\times{\mathcal M}\, \subset\, X\times X\times {\mathcal M}$. Recall from
Section \ref{se3.1} that the restriction of $(p^*_1 K^{1/2}_X)\otimes (p^*_2 K^{1/2}_X) \otimes
{\mathcal O}_{X\times X}(\Delta)$
to $\Delta\, \subset\, X\times X$ is identified with ${\mathcal O}_X$; so, the restriction
of ${\mathcal E}\otimes {\mathcal K}\otimes q^*_{12} {\mathcal O}_{X\times X}(\Delta)$ to
$\Delta\times\mathcal M$ is identified with ${\mathcal E}\big\vert_{\Delta\times
{\mathcal M}}\,=\, \mathcal V$. Now
consider the long exact sequence of direct images, for the projection $J$ in \eqref{J}, corresponding to the
short exact sequence of sheaves in \eqref{ls}:
\begin{equation}\label{s0}
0\, \longrightarrow\, J_*({\mathcal E}\otimes{\mathcal K})\, \longrightarrow\, J_*({\mathcal E}\otimes {\mathcal K}\otimes q^*_{12}
{\mathcal O}_{X\times X}(\Delta))\, \longrightarrow\, J_*{\mathcal V}
\end{equation}
$$
\longrightarrow\, R^1J_*({\mathcal E}\otimes {\mathcal K}) \, \longrightarrow\, \ldots\, .
$$
First note that $J_*({\mathcal E}\otimes {\mathcal K})\,=\, 0$, because for every 
$E\, \in\, {\mathcal M}^0$, we have
$$
H^0(X\times X,\, (p^*_1(E\otimes K^{1/2}_X))\otimes (p^*_2(E^*\otimes K^{1/2}_X)))\,=\, 0
$$
(see \eqref{cv}). Also, $J_*{\mathcal V}\,=\, {\mathcal O}_{\mathcal M}$, because every stable vector bundle is simple.
Consequently, from \eqref{s0} we have the exact sequence
\begin{equation}\label{s1}
0\, \longrightarrow\, J_*({\mathcal E}\otimes {\mathcal K}\otimes q^*_{12}
{\mathcal O}_{X\times X}(\Delta))\, \longrightarrow\, {\mathcal O}_{\mathcal M}
\longrightarrow\, R^1J_*({\mathcal E}\otimes {\mathcal K})\, .
\end{equation}

Next we note that
$$
H^1(X\times X,\, (p^*_1(E\otimes K^{1/2}_X))\otimes (p^*_2(E^*\otimes K^{1/2}_X)))\,=\, 0
$$
for all $E\, \in\, {\mathcal M}^0$ (see \eqref{cv}). Also, for a general point
$E\, \in\, D_\Theta$, using \eqref{cv0} it follows that
$$
\dim H^1(X\times X,\, (p^*_1(E\otimes K^{1/2}_X))\otimes (p^*_2(E^*\otimes K^{1/2}_X)))\,=\, 1\, .
$$
Consequently, the support of $R^1J_*({\mathcal E}\otimes {\mathcal K})$ is the divisor
$D_\Theta$, and the rank of the sheaf $$R^1J_*({\mathcal E}\otimes {\mathcal K})\,
\longrightarrow\, D_\Theta$$ is one.

Let $1_{\mathcal M}$ be the section of ${\mathcal O}_{\mathcal M}$ given by the constant function $1$ on
$\mathcal M$. Since $R^1J_*({\mathcal E}\otimes {\mathcal K})$ is supported on $D_\Theta$, from \eqref{s1}
we conclude the following:
\begin{itemize}
\item The restriction $1_{\mathcal M}$ to
${\mathcal M}^0\,=\, {\mathcal M}\setminus D_\Theta\, \subset\, {\mathcal M}$ is a
holomorphic section of
$$(J_*({\mathcal E}\otimes {\mathcal K}\otimes q^*_{12}
{\mathcal O}_{X\times X}(\Delta)))\big\vert_{{\mathcal M}^0}\, \longrightarrow\, {\mathcal M}^0
$$
(more precisely, $1_{\mathcal M}\big\vert_{{\mathcal M}^0}$ is the image
of a holomorphic section of $(J_*({\mathcal E}\otimes {\mathcal K}\otimes q^*_{12}
{\mathcal O}_{X\times X}(\Delta)))\big\vert_{{\mathcal M}^0}$); this section of
$(J_*({\mathcal E}\otimes {\mathcal K}\otimes q^*_{12}
{\mathcal O}_{X\times X}(\Delta)))\big\vert_{{\mathcal M}^0}$ given by $1_{\mathcal M}$
will be denoted by $1'_{\mathcal M}$.

\item The above defined $1'_{\mathcal M}$ is a meromorphic section of
$J_*({\mathcal E}\otimes {\mathcal K}\otimes q^*_{12}
{\mathcal O}_{X\times X}(\Delta))$ with a pole of order one on $D_\Theta$.
\end{itemize}
In other words, we have
\begin{equation}\label{s2}
1'_{\mathcal M}\, \in\, H^0({\mathcal M},\, J_*({\mathcal E}\otimes {\mathcal K}\otimes q^*_{12}
{\mathcal O}_{X\times X}(\Delta))\otimes {\mathcal O}_{\mathcal M}(D_\Theta))\, .
\end{equation}

Now using the projection formula we have
$$J_*({\mathcal E}\otimes {\mathcal K}\otimes q^*_{12}
{\mathcal O}_{X\times X}(\Delta))\otimes {\mathcal O}_{\mathcal M}(D_\Theta)\,=\,
J_*({\mathcal E}\otimes {\mathcal K}\otimes (q^*_{12}
{\mathcal O}_{X\times X}(\Delta))\otimes J^*{\mathcal O}_{\mathcal M}(D_\Theta))\, ,
$$
and hence $1'_{\mathcal M}$ in \eqref{s2} defines a section
\begin{equation}\label{s3}
1''_{\mathcal M}\, \in\, H^0(X\times X\times{\mathcal M},\,
{\mathcal E}\otimes {\mathcal K}\otimes (q^*_{12}
{\mathcal O}_{X\times X}(\Delta))\otimes J^*{\mathcal O}_{\mathcal M}(D_\Theta))\, .
\end{equation}

For every $E\, \in\, {\mathcal M}^0$, the section $\beta_E$ in \eqref{e8} coincides with 
the restriction $1''_{\mathcal M}\big\vert_{X\times X\times\{E\}}$, where $1''_{\mathcal M}$ is
the section in \eqref{s3}. Now from the construction in \eqref{c1} of the section $\phi$ of the 
$T^*{\mathcal M}^0$--torsor ${\mathcal C}^0$ it follows that
\begin{enumerate}
\item $\phi$ is meromorphic, and 

\item the order of pole, at $D_\Theta$, of $\phi$ is one.
\end{enumerate}
As noted before, Proposition \ref{prop1} follows from these two.
\end{proof}

Continuing with the proof Theorem \ref{thm2}, let
\begin{equation}\label{e10}
\Psi\,\, :\,\, {\mathcal M}\, \longrightarrow\, J(X)\, =\, \text{Pic}^0(X)
\end{equation}
be the determinant map $E\, \longmapsto\, \bigwedge^rE$. The image of the pullback homomorphism
$$
(d\Psi)^*\,\, :\,\, \Psi^*T^*J(X)\, \longrightarrow\, T^*{\mathcal M}\, ,
$$
where $d\Psi$ is the differential of $\Psi$,
has a canonical direct summand; we shall now recall a description of this direct summand.

As in \eqref{cV}, let
${\mathcal V}\, \longrightarrow\, X\times {\mathcal M}$ be the universal endomorphism
bundle, and let $${\mathcal V}^0\,\subset\, \mathcal V$$ be the universal endomorphism
bundle of trace zero. There is a natural decomposition into traceless and trace components:
\begin{equation}\label{bd}
{\mathcal V}\, =\, {\mathcal V}^0\oplus {\mathcal O}_{X\times\mathcal M}\, ;
\end{equation}
the above inclusion map ${\mathcal O}_{X\times\mathcal M}\, \hookrightarrow\, \mathcal V$
is defined by $f\, \longmapsto\, f\cdot\text{Id}$. Let
\begin{equation}\label{P}
P\, :\, X\times {\mathcal M}\, \longrightarrow\, {\mathcal M}\ \ \text{ and }
p\, :\, X\times {\mathcal M}\, \longrightarrow\, X
\end{equation}
be the natural projections. Then we have
$$
T^*{\mathcal M}\,=\, P_* ({\mathcal V}\otimes p^*K_X)\, ,
$$
where $P$ and $p$ are the projections in \eqref{P}.
Consequently, the decomposition in \eqref{bd} produces a holomorphic decomposition
\begin{equation}\label{bd2}
T^*{\mathcal M}\, =\, P_*({\mathcal V}^0\otimes p^*K_X) \oplus
P_*p^*K_X \, ;
\end{equation}
we note that $P_*p^*K_X$ is the trivial holomorphic vector bundle
$$
{\mathcal M}\times H^0(X,\, K_X)\, \longrightarrow\, \mathcal M
$$
with fiber $H^0(X,\, K_X)$. Tensoring \eqref{bd2} with $\Theta$ we obtain
$$
(T^*{\mathcal M})\otimes\Theta\, =\, P_*({\mathcal V}^0\otimes p^*K_X)\otimes\Theta
\oplus ({\mathcal M}\times H^0(X,\, K_X))\otimes\Theta\, .
$$
This produces a decomposition
$$
H^0({\mathcal M},\, (T^*{\mathcal M})\otimes\Theta)\,=\,
H^0({\mathcal M},\, P_*({\mathcal V}^0\otimes p^*K_X)\otimes\Theta)\oplus 
(H^0({\mathcal M},\, \Theta)\otimes H^0(X,\, K_X))
$$
\begin{equation}\label{e11}
=\,
H^0({\mathcal M},\, P_*({\mathcal V}^0\otimes p^*K_X)\otimes\Theta)\oplus 
H^0(X,\, K_X)\, ;
\end{equation}
the last equality follows from the fact that $H^0({\mathcal M},\, \Theta)\,=\,
\mathbb C$ \cite[p.~169, Theorem 2]{BNR}. We note that the inclusion map
$$
H^0(X,\, K_X)\,=\, H^0(J(X),\, T^*J(X)) \, \hookrightarrow\, 
H^0({\mathcal M},\, (T^*{\mathcal M})\otimes\Theta)
$$
in \eqref{e11} coincides with the pullback of $1$--forms on $J(X)$
to ${\mathcal M}$ by the projection $\Psi$ in \eqref{e10}.

The following proposition would be used in the proof of Theorem \ref{thm2}.

\begin{proposition}\label{prop2}
For the projections $P$ and $p$ in \eqref{P},
$$
H^0({\mathcal M},\, P_*({\mathcal V}^0\otimes p^*K_X)\otimes\Theta)\,=\, 0\, ,
$$
where ${\mathcal V}^0$ is the subbundle in \eqref{bd}.
\end{proposition}

\begin{proof}[{Proof of Proposition \ref{prop2}}]
If $r\,=\,1$, then ${\mathcal V}^0\,=\, 0$, and hence in this case the
proposition is obvious. Hence in the proof we assume that $r\, \geq\, 2$. The proof proceeds by showing 
that the sections must vanish on the projective spaces lying inside ${\mathcal M}$ given by Hecke transforms.

Let
\begin{equation}\label{bp}
H\, :\, {\mathbb P}\, \longrightarrow\, X\times{\mathcal M}
\end{equation}
be the universal projective bundle; so for any $(x,\, E)\, \in\, X\times{\mathcal M}$,
the inverse image $H^{-1}(x,\, E)$ is the space of all hyperplanes in the fiber $E_x$;
in particular,
$\mathbb P$ is a holomorphic fiber bundle over $X\times{\mathcal M}$ with
the projective space ${\mathbb C}{\mathbb P}^{r-1}$ as the typical fiber. Let
$$
T_H\,=\, \text{kernel}(dH) \, \longrightarrow\, {\mathbb P}
$$
be the (holomorphic) relative tangent bundle for the projection $H$ in \eqref{bp}, where $dH$
is the differential of the map $H$. We note that ${\mathcal V}^0$ in \eqref{bd} is the
following direct image:
$$
H_* T_H \, =\, {\mathcal V}^0\, \longrightarrow\, X\times{\mathcal M}\, .
$$

Given any element $(x,\, E)\, \in\, X\times \mathcal M$, along with
a hyperplane $S\, \subset\, E_x$,
let $F(x, E,S)$ be the holomorphic vector bundle over $X$ whose sheaf of sections fits
in the short exact sequence of coherent sheaves on $X$
$$
0\, \longrightarrow\, F(x, E,S)\, \longrightarrow\, E\, \longrightarrow\, E_x/S\,
\longrightarrow\, 0\, ;
$$
the above sheaf $E_x/S$ is the torsion sheaf supported
at the point $x$ and its stalk at $x$ is the quotient line $E_x/S$.

Let $\mathcal N$ denote the moduli space of stable vector bundles over $X$ of rank $r$ and
degree $-1$. Using the above construction of $F(x, E,S)$, we get a rational map
$$
\xi\, :\, {\mathbb P}\, \dasharrow\, X\times \mathcal N\, , \ \ (x,\, E,\,S)\,\longmapsto\,
(x,\, F(x, E,S))\, ,
$$
which is called the Hecke morphism \cite{NR1}, \cite{NR2}.
It is known that there is a nonempty Zariski open subset
$${\mathcal U}\, \subset\, {\mathbb P}$$
such that the pair $(\xi,\, {\mathcal U})$ satisfies the following conditions:
\begin{enumerate}
\item The rational map $\xi$ is actually defined as a map on $\mathcal U$; the restriction of
$\xi$ to $\mathcal U$ will be denoted by $\widehat{\xi}$.

\item The codimension of the complement ${\mathbb P}\setminus \mathcal U$ is at least
two (see the proof of \cite[Proposition 5.4]{NR2}).

\item The map $\widehat{\xi}\, :\, {\mathcal U}\, \longrightarrow\, \xi({\mathcal U})$
defines a holomorphic fiber bundle over $\xi({\mathcal U})$
with the projective space ${\mathbb C}{\mathbb P}^{r-1}$ as
the typical fiber \cite[p.~411, Proposition 6.8]{NR2}.

\item The relative tangent bundle $T_H$ on $\mathcal U$ coincides with
$\Omega_{\widehat{\xi}}\otimes (p\circ H)^*K_X$, where $\Omega_{\widehat{\xi}}$ is
the relative cotangent bundle for the map $\widehat{\xi}$
(see \cite[p. 265, (2.7)]{Bi1}).
\end{enumerate}
It may be clarified that the above vector bundle $\Omega_{\widehat{\xi}}$ is the
cokernel of the pullback homomorphism $$d(\widehat{\xi})^*\, :\,
\widehat{\xi}^*T^*(\xi({\mathcal U}))\,\longrightarrow\, T^*{\mathcal U}\, .$$

Take a point
\begin{equation}\label{pz}
z\,=\, (x,\, W)\, \in\, \widehat{\xi}({\mathcal U})\, \subset\,
X\times \mathcal N\, .
\end{equation}
Let
$$
{\mathbb F}_z\,=\, \widehat{\xi}^{-1}(z)\, \subset\, {\mathcal U}
$$
be the fiber of $\widehat{\xi}$ over $z$; as mentioned in
(3) above, this fiber is isomorphic to 
${\mathbb C}{\mathbb P}^{r-1}$. We will compute the restriction of the line bundle $(P\circ 
H)^*\Theta$ to ${\mathbb F}_z \,\cong\,{\mathbb C}{\mathbb P}^{r-1}$, where $P$ and $H$
are the projections in \eqref{P} and \eqref{bp} respectively.

Let $P(W_x)$ be the projective space that parametrizes all the lines in the fiber $W_x$
of the vector bundle $W$ in \eqref{pz}. Let
$L_0\, \longrightarrow\, P(W_x)$ be the tautological line bundle of degree $-1$; the
fiber of $L_0$ over any line $\zeta\, \subset\, W_x$ is $\zeta$ itself.
The inverse image ${\mathbb F}_z\,=\, \widehat{\xi}^{-1}(z)$ is identified with
this projective space $P(W_x)$. For any $\zeta\, \in\, P(W_x)$, the corresponding
element $(x,\, E,\, S)\,\in\, {\mathbb F}_z \, \subset\,{\mathbb P}$ is uniquely determined by
the following condition: The holomorphic vector bundle $E$ fits in the short exact sequence
of sheaves on $X$
$$
0\, \longrightarrow\, W\, \longrightarrow\, E\, \longrightarrow\, {\mathcal Q}\,
\longrightarrow\, 0\, ,
$$
where $\mathcal Q$ is a torsion sheaf of degree one supported at $x$, and the
kernel of the homomorphism of fibers $W_x\, \longrightarrow\, E_x$, given by
the above homomorphism $W\, \longrightarrow\, E$ of sheaves, is the line $\zeta$, while the
subspace $S\, \subset\, E_x$ is the image of this homomorphism $W_x\, \longrightarrow\, E_x$.

To describe the fiber ${\mathbb F}_z$ globally,
let $\Pi_1$ (respectively, $\Pi_2$) be the projection of $X\times P(W_x)$ to $X$
(respectively, $P(W_x)$). On $X\times P(W_x)$ we have the holomorphic vector bundle ${\mathcal W}$
which is defined by the short exact sequence of sheaves on $X\times P(W_x)$
$$
0\,\longrightarrow\, {\mathcal W}^*\,\longrightarrow\, \Pi^*_1 W^*
\,\longrightarrow\, \iota^x_* L^*_0\,\longrightarrow\, 0\, ,
$$
where $\iota^x\, :\, P(W_x)\, \longrightarrow\, X\times P(W_x)$
is the embedding defined by $y\, \longmapsto\, (x,\, y)$. From this exact sequence
it follows that $\mathcal W$ fits in the short exact sequence of sheaves
\begin{equation}\label{ps}
0\,\longrightarrow\, \Pi^*_1 W \,\longrightarrow\, {\mathcal W}
\,\longrightarrow\, \iota^x_* L_0\,\longrightarrow\, 0
\end{equation}
on $X\times P(W_x)$. The map $H\big\vert_{{\mathbb F}_z}$ coincides with the classifying map
$${\mathbb F}_z\, \longrightarrow\, {\mathcal U}\, \subset\, {\mathcal M}$$ for the
above holomorphic family of vector bundles ${\mathcal W}$ on $X$ parametrized by $P(W_x)\,=\, {\mathbb F}_z$.

Now, tensoring the exact sequence in \eqref{ps} with $\Pi^*_1 K^{1/2}_X$,
and then taking the long exact sequence of direct images with respect to the
projection $\Pi_2$, we have the exact sequence of sheaves on $P(W_x)$
$$
0\,\longrightarrow\, \Pi_{2*} (\Pi^*_1 (W\otimes K^{1/2}_X))
\,\longrightarrow\,\Pi_{2*} ({\mathcal W}\otimes \Pi^*_1 K^{1/2}_X)
\,\longrightarrow\, L_0
$$
\begin{equation}\label{les}
\,\longrightarrow\, R^1 \Pi_{2*} (\Pi^*_1 (W\otimes K^{1/2}_X)) 
\,\longrightarrow\, R^1 \Pi_{2*} ({\mathcal W}\otimes \Pi^*_1 K^{1/2}_X)
\,\longrightarrow\, 0\, ;
\end{equation}
note that the restriction of $\Pi^*_1 K^{1/2}_X$ to the image of the embedding $\iota^x$
is a trivial line bundle, and also note that $$R^1\Pi_{2*} ((\iota^x_* L_0)\otimes \Pi^*_1 K^{1/2}_X)
\,=\, 0$$
because the support of $(\iota^x_* L_0)\otimes \Pi^*_1 K^{1/2}_X$ is finite over
$P(W_x)$.

Since $H\big\vert_{{\mathbb F}_z}$ coincides with the classifying map for the
above holomorphic family of vector bundles ${\mathcal W}$ on $X$ parametrized by $P(W_x)$, it follows
that the pulled back line bundle $((P\circ H)^*\Theta)\big\vert_{{\mathbb
F}_z}$, where $P$ is the projection in \eqref{P},
is identified with the line bundle
\begin{equation}\label{di}
((P\circ H)^*\Theta)\big\vert_{{\mathbb
F}_z}\,=\,
\det (\Pi_{2*} ({\mathcal W}\otimes \Pi^*_1 K^{1/2}_X))^*\otimes
\det (R^1 \Pi_{2*} ({\mathcal W}\otimes \Pi^*_1 K^{1/2}_X))
\end{equation}
(see \cite[Ch.~V, \S~6]{Ko} for the construction of determinant bundle). For any exact sequence of
coherent sheaves
$$
0\,\longrightarrow\, A_1 \,\longrightarrow\, A_2 \,\longrightarrow\, \ldots \,\longrightarrow\,
A_m \,\longrightarrow\, 0
$$
on a complex manifold $Y$, we have $\bigotimes_{i=1}^m (\det (A_i))^{(-1)^i}\,=\, {\mathcal O}_Y$
\cite[p.~165, Proposition (6.9)]{Ko}. Consequently, from
\eqref{les} and \eqref{di} we conclude that
$$
((P\circ H)^*\Theta)\big\vert_{{\mathbb F}_z}\,=\, L^*_0\, ,
$$
because both $\Pi_{2*} (\Pi^*_1 (W\otimes K^{1/2}_X))$ and
$R^1 \Pi_{2*} (\Pi^*_1 (W\otimes K^{1/2}_X))$ are trivial vector bundles.
In other words, the degree of the line bundle $(P\circ H)^*\Theta$ restricted
to ${\mathbb F}_z\,=\, P(W_x)$ is $1$.

Using the above properties of $(\xi,\, {\mathcal U})$ we are in a position to complete the
proof of the proposition.

We have
\begin{equation}\label{f1}
H^0({\mathcal M},\, P_*({\mathcal V}^0\otimes p^*K_X)\otimes\Theta)\,=\,
H^0(X\times {\mathcal M},\, {\mathcal V}^0\otimes (p^*K_X)\otimes (P^*\Theta))\, ,
\end{equation}
because $P_*({\mathcal V}^0\otimes (p^*K_X)\otimes (P^*\Theta))\,=\,
P_*({\mathcal V}^0\otimes p^*K_X)\otimes\Theta$ by the projection formula. Next
we have
$$
H^0(X\times {\mathcal M},\, {\mathcal V}^0\otimes (p^*K_X)\otimes (P^*\Theta))\, =\,
H^0({\mathbb P},\, T_H\otimes ((p\circ H)^*K_X)\otimes (P\circ H)^*\Theta)
$$
\begin{equation}\label{f2}
=\, 
H^0({\mathcal U},\, T_H\otimes ((p\circ H)^*K_X)\otimes (P\circ H)^*\Theta)\, ;
\end{equation}
the first equality follows from the fact that $H_*(T_H\otimes ((p\circ H)^*K_X)\otimes (P\circ H)^*\Theta)\,=\,
{\mathcal V}^0\otimes (p^*K_X)\otimes (P^*\Theta)$ (by the projection formula), and the second equality
follows from the fact that the codimension of the complement ${\mathbb P}\setminus \mathcal U$ is at least two.

As before, take a fiber ${\mathbb F}_z\,=\, \widehat{\xi}^{-1}(z)$ of the map $\widehat{\xi}$.
As shown above, ${\mathbb F}_z$ is identified with the projective space $P(W_x)$, and 
the restriction of $T_H$ (respectively, $(P\circ H)^*\Theta$)
to ${\mathbb F}_z$
is isomorphic to $T^*{\mathbb F}_z$ (respectively, ${\mathcal O}_{{\mathbb F}_z}(1)$);
note that the restriction of $(p\circ H)^*K_X$ to $\widehat{\xi}^{-1}(z)$ is a trivial
line bundle. Consequently, the
restriction of $T_H\otimes ((p\circ H)^*K_X)\otimes (P\circ H)^*\Theta$
to ${\mathbb F}_z$
is isomorphic to $T^*{\mathbb F}_z\otimes{\mathcal O}_{{\mathbb F}_z}(1)$.

Next we note that the holomorphic vector bundle $T^*{\mathbb F}_z\otimes
{\mathcal O}_{{\mathbb F}_z}(1)$ on the projective space
${\mathbb F}_z\,=\, P(W_x)$
is semistable of negative degree (its degree is $-1$), and hence the
vector bundle $T^*{\mathbb F}_z\otimes {\mathcal O}_{{\mathbb F}_z}(1)$
does not have any nonzero holomorphic section. This implies that
$$
H^0({\mathcal U},\, T_H\otimes ((p\circ H)^*K_X)\otimes (P\circ H)^*\Theta)\, =\, 0\, .
$$
Consequently, from \eqref{f2} and \eqref{f1} we now conclude that
$$
H^0({\mathcal M},\, P_*({\mathcal V}^0\otimes p^*K_X)\otimes\Theta)\,=\,0\, .
$$
This completes the proof of Proposition \ref{prop2}.
\end{proof}

We continue with the proof of Theorem \ref{thm2}.
Combining \eqref{e11} with Proposition \ref{prop2}, it follows that we are reduced to the trace component:
\begin{equation}\label{e12}
H^0({\mathcal M},\, (T^*{\mathcal M})\otimes\Theta)\, =\,
\{\Psi^*\omega\, \mid\, \omega\,\in\, H^0(J(X),\, T^*J(X))\}\,=\,
H^0(X,\, K_X)\, ,
\end{equation}
where $\Psi$ is the projection in \eqref{e10}.

Let
\begin{equation}\label{e13}
\Gamma'\, \in\, H^0(J(X),\, T^*J(X))\,=\, H^0(X,\, K_X)
\end{equation}
be the $1$--form corresponding to the section $\Gamma$ in Proposition \ref{prop1} for the
isomorphism in \eqref{e12}.

The proof of Theorem \ref{thm2} will be completed using the following lemma.

\begin{lemma}\label{lem2}
The $1$--form $\Gamma'$ on $J(X)$ in \eqref{e13} is invariant under the holomorphic involution
$$\iota_J\, :\, J(X)\, \longrightarrow\, J(X)$$ defined by $L\, \longmapsto\, L^*$.
\end{lemma}

\begin{proof}[{Proof of Lemma \ref{lem2}}]
Let $\iota_{\mathcal M}\, :\, {\mathcal M}\, \longrightarrow\,{\mathcal M}$
be the holomorphic involution defined by $E\, \longmapsto\, E^*$. Note that
$$
\iota_J\circ \Psi\,=\, \Psi\circ\iota_{\mathcal M}\, ,
$$
where $\Psi$ is constructed in \eqref{e10} and $\iota_J$ is defined in the
statement of the lemma. By Serre duality,
\begin{equation}\label{sdl}
H^k(X,\, E^*\otimes K^{1/2}_X)\,=\, H^{1-k}(X,\, E\otimes K^{1/2}_X)^*
\end{equation}
for $k\,=\, 0,\, 1$. This implies that the above involution $\iota_{\mathcal M}$ preserves
the divisor $D_\Theta$ defined in \eqref{td}. Since $D_\Theta$ is preserved by
$\iota_{\mathcal M}$, the involution $\iota_{\mathcal M}$ has a tautological lift to the
line bundle $\Theta\,=\, {\mathcal O}_{\mathcal M}(D_\Theta)$. Let
\begin{equation}\label{l1}
\iota_\Theta\, \, :\,\, \Theta\, \longrightarrow\, \Theta
\end{equation}
be the resulting involution of $\Theta$ over the involution $\iota_{\mathcal M}$ of $\mathcal M$.
This involution $\iota_\Theta$ of $\Theta$ produces a holomorphic
involution of the complex manifold ${\rm Conn}(\Theta)$ constructed in \eqref{e3}.
The involution of ${\rm Conn}(\Theta)$ constructed this way will be denoted by $\iota_{T}$. We note that
$$
{\iota}_{\mathcal M}\circ q\,=\, q\circ \iota_{T}\, ,
$$
where $q$ is the projection in \eqref{e3}. This implies that the involution $\iota_T$ preserves the open subset
${\rm Conn}(\Theta)^0\,=\, q^{-1}({\mathcal M}\setminus D_\Theta)$ in \eqref{c0c}.

The involution $\iota_T$ (respectively, ${\iota}_{\mathcal M}$) defines an action of
${\mathbb Z}/2\mathbb Z$ on ${\rm Conn}(\Theta)$ (respectively, ${\mathcal M}$).
The section $\tau$ in \eqref{c1c} is evidently ${\mathbb Z}/2\mathbb Z$--equivariant, for the actions
of ${\mathbb Z}/2\mathbb Z$ on ${\rm Conn}(\Theta)^0$ and ${\mathcal M}^0$.

For any $E\,\in\, \mathcal M$, the fiber $\Theta_E$ of $\Theta$ over $E$ is the line 
$\bigwedge^{\rm top}H^0(X,\, E\otimes K^{1/2}_X)^*\otimes \bigwedge^{\rm top}H^1(X,\, E\otimes 
K^{1/2}_X)$. Using \eqref{sdl} we get an isomorphism of $\Theta_E$ with the fiber $\Theta_{E^*}$. 
Also, the involution $\iota_\Theta$ of $\Theta$ in \eqref{l1} produces an isomorphism of $\Theta_E$ 
with $\Theta_{E^*}$. These two isomorphisms between $\Theta_E$ and $\Theta_{E^*}$ actually 
coincide.

The K\"ahler form $\omega_{\mathcal M}$ on $\mathcal M$ (see \eqref{kf}) is clearly 
preserved by the involution $\iota_{\mathcal M}$ of $\mathcal M$. From this it can be deduced that 
the section $\psi_Q$ in \eqref{e5} is ${\mathbb Z}/2\mathbb Z$--equivariant, for the above actions of 
${\mathbb Z}/2\mathbb Z$ on ${\rm Conn}(\Theta)$ and ${\mathcal M}$. Indeed, the section $\psi_Q$ 
corresponds to the unique Hermitian connection on $\Theta$ whose curvature is the K\"ahler form 
$\omega_{\mathcal M}$. In other words, the section $\psi_Q$ is uniquely determined by 
$\omega_{\mathcal M}$. Therefore, the section $\psi_Q$ in \eqref{e5} is
${\mathbb Z}/2\mathbb Z$--equivariant, because $\omega_{\mathcal M}$ is preserved
by $\iota_{\mathcal M}$.

Given a holomorphic connection $\nabla$ on a holomorphic vector bundle $E$, the dual
vector bundle $E^*$ is equipped with the dual connection $\nabla^*$. Therefore, we have
a holomorphic involution
$$
\iota_{\mathcal C}\, :\, {\mathcal C}\, \longrightarrow\, {\mathcal C}\, ,\ \
(E,\, \nabla)\, \longmapsto\, (E^*,\, \nabla^*)\, .
$$
The involution $\iota_{\mathcal C}$ gives an action of ${\mathbb Z}/2\mathbb Z$ on $\mathcal C$.
The projection $\varphi$ in \eqref{e1} is clearly ${\mathbb Z}/2\mathbb Z$--equivariant, for
the actions of ${\mathbb Z}/2\mathbb Z$ on ${\mathcal C}$ and ${\mathcal M}$.
In particular, $\iota_{\mathcal C}$ preserves the Zariski open subset
${\mathcal C}^0$ in \eqref{c0}. Since the dual of a unitary connection on $E$ is a unitary
connection on $E^*$, the section $\psi_U$ in \eqref{e2} is ${\mathbb Z}/2\mathbb Z$--equivariant, for
the actions of ${\mathbb Z}/2\mathbb Z$ on ${\mathcal C}$ and ${\mathcal M}$.

Let $\widehat{\iota}\, :\, X\times X\times {\mathcal M}\, \longrightarrow\, X\times X\times
{\mathcal M}$ be the holomorphic involution defined by $(x,\, y,\, E)\,\longmapsto\, (y,\, x,\,
\iota_{\mathcal M}(E))\,=\, (y,\, x,\, E^*)$. This involution naturally lifts to an involution
of the vector bundle ${\mathcal E}\otimes {\mathcal K}\otimes
q^*_{12}{\mathcal O}_{X\times X}(\Delta)$ in \eqref{cvb}. The earlier mentioned
involution $\iota_\Theta$ of the line bundle $\Theta$
produces an action of ${\mathbb Z}/2\mathbb Z$ on the pullback $J^*\Theta$,
where $J$ is the projection in \eqref{J}. These actions of ${\mathbb Z}/2\mathbb Z$ on
${\mathcal E}\otimes {\mathcal K}\otimes q^*_{12}{\mathcal O}_{X\times X}(\Delta)$ and
$J^*\Theta$ together produce an action of ${\mathbb Z}/2\mathbb Z$ on the tensor product
$$
{\mathcal E}\otimes {\mathcal K}\otimes (q^*_{12}{\mathcal O}_{X\times X}(\Delta))
\otimes J^*\Theta\, \longrightarrow\, X\times X\times{\mathcal M}\, .
$$
The section $1''_{\mathcal M}$ in \eqref{s3} of this tensor product is anti-invariant for the above
action of ${\mathbb Z}/2\mathbb Z$ on ${\mathcal E}\otimes {\mathcal K}\otimes (q^*_{12}{\mathcal
O}_{X\times X}(\Delta))\otimes J^*\Theta$ (meaning the nontrivial element of 
${\mathbb Z}/2\mathbb Z$ acts as multiplication by $-1$). From this it follows that the 
section $\phi$ in \eqref{c1} is ${\mathbb Z}/2\mathbb Z$--equivariant, for the actions of
${\mathbb Z}/2\mathbb Z$ on ${\mathcal C}^0$ and ${\mathcal M}^0$.

{}From all these it follows that $\Gamma_0$ (constructed in \eqref{ga0}) is ${\mathbb Z}/2\mathbb 
Z$--equivariant, for the actions of ${\mathbb Z}/2\mathbb Z$ on ${\mathcal C}^0$ and $T^*{\mathcal M}^0$; the 
action of ${\mathbb Z}/2\mathbb Z$ on $T^*{\mathcal M}^0$ is induced by the action of ${\mathbb Z}/2\mathbb Z$ 
on ${\mathcal M}^0$ constructed using the above involution $\iota_{\mathcal M}$. Since $\Gamma_0$ is ${\mathbb 
Z}/2\mathbb Z$--equivariant, it follows that $\Gamma$ in Proposition \ref{prop1} is ${\mathbb Z}/2{\mathbb 
Z}$--invariant for the action on $(T^*{\mathcal M})\otimes\Theta$ constructed using the actions of ${\mathbb 
Z}/2{\mathbb Z}$ on $T^*\mathcal M$ and $\Theta$ (given by $\iota_\Theta$ in \eqref{l1}). This immediately 
implies that $\Gamma'$ in \eqref{e13} is left invariant under the involution $\iota_J$ of $J(X)$. This completes 
the proof of Lemma \ref{lem2}.
\end{proof}

Continuing with the proof of Theorem \ref{thm2}, we note that $\iota^*_J\alpha\, =\,-\alpha$ for
all $\alpha\, \in\, H^0(J(X),\, T^*J(X))$, where $\iota_J$ is the involution
in Lemma \ref{lem2}. Hence from Lemma \ref{lem2} it follows immediately that 
$\Gamma'\,=\, 0$. In view of \eqref{e12}, this implies that $\Gamma$ in Proposition 
\ref{prop1} vanishes identically. Hence $\Gamma_0$ in \eqref{ga0} vanishes identically.
Therefore, we conclude that the restriction, to the 
open subset ${\mathcal C}^0\, \subset\, \mathcal C$, of the isomorphism $F$ in Theorem \ref{thm1}
coincides with the isomorphism $G$ in Lemma \ref{lem1}. This completes the proof
of Theorem \ref{thm2}.
\end{proof}

Theorem \ref{thm1} and Theorem \ref{thm2} together give the following:

\begin{corollary}\label{cor0}
The holomorphic isomorphism
$G\, :\, {\mathcal C}^0\, \longrightarrow\, {\rm Conn}(\Theta)^0$
in Lemma \ref{lem1} extends to a holomorphic isomorphism
$$G'\, :\, {\mathcal C}\, \stackrel{\sim}{\longrightarrow}\, {\rm Conn}(\Theta)\, .$$
\end{corollary}

\begin{proof}
Since $F$ in Theorem \ref{thm1} is a holomorphic isomorphism from
${\mathcal C}$ to ${\rm Conn}(\Theta)$, this follows from
Theorem \ref{thm2}.
\end{proof}

The isomorphism $G'$ in Corollary \ref{cor0} has the following property:

\begin{corollary}\label{cor2}
For the isomorphism $G'$ in Corollary \ref{cor0},
$$(G')^*\Phi_2\,=\, 2r\cdot\Phi_1\, ,$$
where $\Phi_1$ and $\Phi_2$ are the symplectic
forms in \eqref{sf} and \eqref{sf2} respectively.
\end{corollary}

\begin{proof}
The final part of Theorem \ref{thm1} says that $F^*\Phi_2\,=\, 2r\cdot\Phi_1$. Since
$F\,=\, G'$, the result follows from this.
\end{proof}

\begin{corollary}\label{cor1}
The image of the section $\phi$ in \eqref{c1} is a Lagrangian submanifold of
${\mathcal C}^0$ equipped with the symplectic form
$\Phi_1\big\vert_{{\mathcal C}^0}$ in \eqref{sf}.
\end{corollary}

\begin{proof}
In Corollary \ref{cor2} we saw that $G'$ is symplectic structure preserving 
(up to the factor $2r$). The image of the section $\tau$ in \eqref{c1c} is clearly a Lagrangian 
submanifold of $\text{Conn}(\Theta)^0$ with respect to the symplectic form $\Phi_2\big\vert_{{\rm 
Conn} (\Theta)^0}$ in \eqref{sf2} (the trivial connection is flat). Since $G(\phi({\mathcal M}^0))\,
=\, \tau({\mathcal M}^0)$, and $G$ is symplectic structure preserving up to the factor $2r$, from the 
above observation --- that the image of $\tau$ is a Lagrangian submanifold of $\text{Conn}(\Theta)^0$ 
with respect to the symplectic form $\Phi_2\big\vert_{{\rm Conn} (\Theta)^0}$ --- it follows 
immediately that $\phi({\mathcal M}^0)$ is a Lagrangian submanifold of ${\mathcal C}^0$ with 
respect to the symplectic form $\Phi_1\big\vert_{{\mathcal C}^0}$ in \eqref{sf}.
\end{proof}

\section{Family of Riemann surfaces}

Let $\mathcal T$ be a connected complex manifold, and let
$$F\, :\, {\mathcal X}_{\mathcal T}\, \longrightarrow\, \mathcal T$$
be a holomorphic family of compact connected Riemann
surfaces of genus $g$, with $g\, \geq\, 2$, parametrized by
$\mathcal T$, and equipped with a theta characteristic ${\mathbb L}$.
This means that ${\mathbb L}$ is a holomorphic line bundle over
${\mathcal X}_{\mathcal T}$, and there is a given holomorphic isomorphism
$$
I\, :\, {\mathbb L}\otimes {\mathbb L}\, \longrightarrow\, K_F\, ,
$$
where $K_F\, \longrightarrow\, {\mathcal X}_{\mathcal T}$ is the relative holomorphic cotangent bundle for the
project $F$; in other words, $K_F$ is the cokernel of the dual of the differential $dF$
$$(dF)^*\, :\, F^*T^*{\mathcal T}\, \longrightarrow\, T^*{\mathcal X}_{\mathcal T}\, .$$
For each point $t\, \in\,{\mathcal T}$, the compact Riemann surface $F^{-1}(t)$ will be denoted by
${\mathcal X}_t$. The holomorphic line bundle ${\mathbb L}\big\vert_{{\mathcal X}_t}$ on ${\mathcal X}_t$
will be denoted by ${\mathbb L}_t$.

Let
$$
\gamma\, :\, {\mathcal M}_{\mathcal T}\, \longrightarrow\, {\mathcal T}
$$
be the relative moduli space of stable vector bundles of rank $r$ and degree zero. So for any
$t\, \in\, \mathcal T$, the fiber $\gamma^{-1}(t)$ is the moduli space of stable vector bundles on
${\mathcal X}_t$ of rank $r$ and degree zero. Let
$$
\Theta_{\mathcal T}\, \longrightarrow\, {\mathcal M}_{\mathcal T}
$$
be the relative theta bundle constructed using the relative theta 
characteristic $\mathbb L$. So $\Theta_{\mathcal T}$ corresponds to the 
reduced effective divisor on ${\mathcal M}_{\mathcal T}$ defined by all $(t,\, E)$, 
where $t\, \in\, {\mathcal T}$ and $E\, \in\, \gamma^{-1}(t)$, such that 
$H^0({\mathcal X}_t,\, E\otimes {\mathbb L}_t)\, \not=\, 0$.

Let
\begin{equation}\label{hfb1}
q_{\mathcal T}\, :\, \text{Conn}^r(\Theta_{\mathcal T}) \, \longrightarrow\, 
{\mathcal M}_{\mathcal T}
\end{equation}
be the holomorphic fiber bundle over ${\mathcal M}_{\mathcal T}$ defined by the sheaf of
relative holomorphic connections
on $\Theta_{\mathcal T}$. So for any $t\, \in\, \mathcal T$, the fiber $(q_{\mathcal T})^{-1}(t)$ is
$\text{Conn}(\Theta)$ in \eqref{e3} for $X\,=\, {\mathcal X}_t$. The holomorphic
fiber bundle in \eqref{hfb1} has a $C^\infty$ section
\begin{equation}\label{hpq}
\widehat{\psi}_Q\, :\, {\mathcal M}_{\mathcal T}\,\longrightarrow\,
\text{Conn}^r(\Theta_{\mathcal T})
\end{equation}
given by the Chern connection associated to the Quillen metric on $\Theta_{\mathcal T}$
\cite{Qu}; so for each $t\, \in\, \mathcal T$, the restriction of $\widehat{\psi}_Q$ to
$\gamma^{-1}(t)$ is the section $\psi_Q$ in \eqref{e5} for the Riemann surface $X\,=\, {\mathcal X}_t$.

Let
\begin{equation}\label{hfb2}
{\varphi}_{\mathcal T}\, :\, {\mathcal C}_{\mathcal T} \,
\longrightarrow\, {\mathcal M}_{\mathcal T}
\end{equation}
be the moduli space of relative holomorphic connections; the fiber of
${\varphi}_{\mathcal T}$ over any $(t,\, E)$, where $t\, \in\, {\mathcal T}$
and $E\, \in\, \gamma^{-1}(t)$, is the space of all holomorphic connections
on the stable vector bundle $E\,\longrightarrow\, {\mathcal X}_t$, in particular,
this fiber is an affine space for $H^0({\mathcal X}_t,\, \text{End}(E)
\otimes T^*{\mathcal X}_t)$.

The holomorphic fiber bundle in \eqref{hfb2} has a $C^\infty$ section
\begin{equation}\label{hpu}
\widehat{\psi}_U\,\, :\,\, {\mathcal M}_{\mathcal T}\,\longrightarrow\,
{\mathcal C}_{\mathcal T}
\end{equation}
that sends any stable vector bundle of degree zero to the unique
holomorphic connection on it whose monodromy is unitary; so for each
$t\, \in\, \mathcal T$, the restriction of $\widehat{\psi}_U$ to $\gamma^{-1}(t)$
is the section $\psi_U$ in \eqref{e2} for the Riemann surface $X\,=\, {\mathcal X}_t$.

For each $t\, \in\, {\mathcal T}$, there is a natural holomorphic isomorphism
$$
F_t\,\, :\,\, (\gamma\circ{\varphi}_{\mathcal T})^{-1}(t)
\, \longrightarrow\,(\gamma\circ q_{\mathcal T})^{-1}(t)
$$
(${\varphi}_{\mathcal T}$ and $q_{\mathcal T}$ are the projections in
\eqref{hfb2} and \eqref{hfb1} respectively)
that takes the image of the section $\widehat{\psi}_U$ (constructed in
\eqref{hpu}) to the image of the section $\widehat{\psi}_Q$ constructed
in \eqref{hpq} (see Theorem \ref{thm1}). These isomorphisms $\{F_t\}_{t\in
\mathcal T}$ together define a $C^\infty$ isomorphism
\begin{equation}\label{wft}
\widehat{F}\,\, :\,\, {\mathcal C}_{\mathcal T}\, \longrightarrow\,
\text{Conn}^r(\Theta_{\mathcal T})\, ;
\end{equation}
the restriction of $\widehat{F}$ to $(\gamma\circ{\varphi}_{\mathcal T})^{-1}(t)$
is the above holomorphic isomorphism $F_t$ for every $t\,\in\,\mathcal T$.

\begin{proposition}\label{prop3}
The $C^\infty$ isomorphism $\widehat{F}$ in \eqref{wft} is holomorphic.
\end{proposition}

\begin{proof}
For every $t\, \in\, \mathcal T$, consider the holomorphic isomorphism
$$
G'_t\,\, :\,\, (\gamma\circ{\varphi}_{\mathcal T})^{-1}(t)
\, \longrightarrow\,(\gamma\circ q_{\mathcal T})^{-1}(t)
$$
in Corollary \ref{cor0}; so $G'_t$ is $G'$ in Corollary \ref{cor0}
for ${\mathcal X}_t\,=\, X$. These isomorphisms combine together to define an isomorphism
$$
\widehat{G}'\,\, :\,\, {\mathcal C}_{\mathcal T}\, \longrightarrow\,
\text{Conn}^r(\Theta_{\mathcal T})\, ;
$$
the restriction of $\widehat{G}'$ to $(\gamma\circ{\varphi}_{\mathcal T})^{-1}
(t)$ is the above holomorphic isomorphism $G'_t$ for every $t\,\in\,\mathcal T$. From
the construction of the isomorphism $G$ in Lemma \ref{lem1} it follows immediately
that $G$ depends holomorphically on the Riemann surface $X$. Note that both the sections
$\phi$ and $\tau$, constructed in \eqref{c1} and \eqref{c1c} respectively, depend holomorphically
on the Riemann surface. Therefore, its extension
$G'$ also depends holomorphically on the Riemann surface $X$. Consequently, the
above isomorphism $\widehat{G}'$ is holomorphic.

Now, Theorem \ref{thm2} implies that $\widehat{G}'$ coincides with $\widehat{F}$ in \eqref{wft}.
Hence the map $\widehat{F}$ is holomorphic.
\end{proof}

\section*{Acknowledgements}

The first-named author thanks Centre de Recherches Math\'ematiques, Montreal, for hospitality. He is partially 
supported by a J. C. Bose Fellowship. The second-named author is supported by an NSERC Discovery grant.



\begin{thebibliography}{AAAA}

\bibitem[At]{At} M. F. Atiyah, Complex analytic connections in fibre
bundles, \textit{Trans. Amer. Math. Soc.} \textbf{85} (1957), 181--207.

\bibitem[AB]{AB} M. F. Atiyah and R. Bott, The Yang-Mills equations over Riemann
surfaces, {\it Philos. Trans. Roy. Soc. London Ser. A} {\bf 308} (1983), 523--615.

\bibitem[BNR]{BNR} A. Beauville, M. S. Narasimhan and S. Ramanan, Spectral curves and
the generalised theta divisor, {\it Jour. Reine Angew. Math.} {\bf 398} (1989), 169--179.

\bibitem[BB]{BB} D. Ben-Zvi and I. Biswas, Theta functions and Szeg\H{o} kernels,
{\it Int. Math. Res. Not.} (2003), no. 24, 1305--1340.

\bibitem[BGS1]{BGS1} J.-M. Bismut, H. Gillet and C. Soul\'e, Analytic torsion and holomorphic
determinant bundles. I. Bott-Chern forms and analytic torsion, {\it Comm. Math. Phys.}
{\bf 115} (1988), 49--78.

\bibitem[BGS2]{BGS2} J.-M. Bismut, H. Gillet and C. Soul\'e, Analytic torsion and
holomorphic determinant bundles. II. Direct images and Bott-Chern forms, {\it Comm. Math. Phys.}
{\bf 115} (1988), 79--126.

\bibitem[BGS3]{BGS3} J.-M. Bismut, H. Gillet and C. Soul\'e, Analytic torsion and holomorphic
determinant bundles. III. Quillen metrics on holomorphic determinants, {\it Comm. Math. Phys.}
{\bf 115} (1988), 301-351.

\bibitem[Bi1]{Bi1} I. Biswas, Infinitesimal deformations of the tangent bundle of a
moduli space of vector bundles over a curve, {\it Osaka Jour. Math.} {\bf 43} (2006),
263--274. 

\bibitem[Bi2]{Bi2} I. Biswas, On the moduli space of holomorphic $G$-connections on a compact
Riemann surface, {\it Euro. Jour. Math.} {\bf 6} (2020), 321--335.

\bibitem[BH]{BH} I. Biswas and J. Hurtubise, Meromorphic
connections, determinant line bundles and the Tyurin parametrization, arXiv:1907.00133,
\textit{Asian Jour. Math.} (to appear).

\bibitem[BHS]{BHS} I. Biswas, J. Hurtubise and J. Stasheff, A construction of a 
universal connection, {\it Forum Math.} {\bf 24} (2012), 365--378.

\bibitem[De]{De} P. Deligne, {\it \'Equations diff\'erentielles \`a points
singuliers r\'eguliers}, Lecture Notes in Mathematics, Vol. 163, Springer-Verlag,
Berlin-New York, 1970.

\bibitem[Go]{Go} W. M. Goldman, The symplectic nature of fundamental groups of surfaces, {\it Adv.
in Math.} {\bf 54} (1984), 200--225.

\bibitem[Ko]{Ko} S. Kobayashi, \textit{Differential geometry
of complex vector bundles}, Princeton University Press, Princeton,
NJ, Iwanami Shoten, Tokyo, 1987.

\bibitem[La]{La} Y. Laszlo, Un th\'eor\`eme de Riemann pour les diviseurs th\^eta sur les espaces
de modules de fibr\'es stables sur une courbe, {\it Duke Math. Jour.} {\bf 64} (1991), 333--347.

\bibitem[Ma]{Ma} B. Malgrange, 
 D\'eformations isomonodromiques, forme de Liouville, fonction tau,
{\it Ann. Institut Fourier} {\bf 54}, (2004), 1371--1392.

\bibitem[NR1]{NR1} M. S. Narasimhan and S. Ramanan, Deformations of the moduli space of 
vector bundles over an algebraic curve, {\it Ann. of Math.} {\bf 101} (1975), 391--417.

\bibitem[NR2]{NR2} M. S. Narasimhan and S. Ramanan, Geometry of Hecke cycles. I,
{\it C. P. Ramanujam--a tribute}, pp. 291--345, Tata Inst. Fund. Res. Studies in Math., 8,
Springer, Berlin-New York, 1978. 

\bibitem[NS]{NS} M. S. Narasimhan and C. S. Seshadri, Stable and unitary vector
bundles on a compact {R}iemann surface, {\em Ann. of Math.} {\bf 82} (1965), 540--567.

\bibitem[Ne]{Ne} P. E. Newstead, {\it Introduction to moduli problems and orbit spaces},
Tata Institute of Fundamental Research Lectures on Mathematics and Physics, 51, Narosa 
Publishing House, New Delhi, 1978.

\bibitem[Qu]{Qu} D. G. Quillen, Determinants of Cauchy-Riemann operators on
Riemann surfaces, \textit{Funct. Anal. Appl.} \textbf{19} (1985), 37--41.

\bibitem[Ra]{Ra} S. Ramanan, The moduli spaces of vector bundles over an algebraic curve,
{\it Math. Ann.} \textbf{200} (1973), 69--84.

\bibitem[Si1]{Si1} C. T. Simpson, Moduli of representations of the fundamental group of a
smooth projective variety. I, {\it Inst. Hautes \'Etudes Sci. Publ. Math.} {\bf 80} (1994),
5--79.

\bibitem[Si2]{Si2} C. T. Simpson, Moduli of representations of the fundamental group of a
smooth projective variety. II, {\it Inst. Hautes \'Etudes Sci. Publ. Math.}
{\bf 79} (1994), 47--129.

\bibitem[We]{We} A. Weil, G\'en\'eralisation des fonctions ab\'eliennes, {\it Jour. Math.
Pures Appl.} {\bf 17} (1938), 47--87.

\end{thebibliography}
\end{document}